\RequirePackage[l2tabu, orthodox]{nag}
\documentclass[a4paper]{amsart}

\usepackage[T1,T2A]{fontenc}
\usepackage[tt/osf=false]{cfr-lm}

\newcommand{\cyrL}{\mbox{\usefont{T2A}{\rmdefault}{m}{n}\CYRL}}

\usepackage{amsmath}
\usepackage{amsthm}
\usepackage{amssymb}
\usepackage[foot]{amsaddr}
\usepackage{mathtools}

\usepackage{tikz}
\usepackage{xcolor}
\usepackage{graphicx}
\usepackage{pinlabel}
\usepackage{booktabs}

\usepackage{subcaption}

\DeclareCaptionLabelFormat{xx}{#2}
\usepackage{enumitem}
\usepackage[backend=bibtex,autocite=inline, maxnames=8, isbn=false, sortcites=true, backref]{biblatex}
\AtEveryBibitem{\clearlist{language}}
\usepackage{hyperref}
\hypersetup{
    colorlinks,
    linkcolor={red!50!black},
    citecolor={blue!50!black},
    urlcolor={blue!80!black}
}

\usepackage{zref-clever}
\usepackage{microtype}

\numberwithin{equation}{section}
\zcsetup{countertype={thm=theorem}}
\newtheorem{thm}{Theorem}
\numberwithin{thm}{section}
\newtheorem{mainthm}{Main Theorem}

\zcRefTypeSetup{mainthm}{
Name-sg = Main Theorem ,
name-sg = main theorem ,
Name-pl = Main Theorems ,
name-pl = main theorems ,
}
\AddToHook{env/lem/begin}{\zcsetup{countertype={thm=lemma}}}
\newtheorem{lem}[thm]{Lemma}
\AddToHook{env/prp/begin}{\zcsetup{countertype={thm=proposition}}}
\newtheorem{prp}[thm]{Proposition}
\AddToHook{env/cor/begin}{\zcsetup{countertype={thm=corollary}}}
\newtheorem{cor}[thm]{Corollary}
\AddToHook{env/prop/begin}{\zcsetup{countertype={thm=proposition}}}
\theoremstyle{definition}
\AddToHook{env/defn/begin}{\zcsetup{countertype={thm=definition}}}
\newtheorem{defn}[thm]{Definition}
\AddToHook{env/ex/begin}{\zcsetup{countertype={thm=example}}}
\newtheorem{ex}[thm]{Example}
\newtheorem{cons}[thm]{Construction}
\zcRefTypeSetup{cons}{
Name-sg = Construction ,
name-sg = construction ,
Name-pl = Constructions ,
name-pl = constructions ,
}
\AddToHook{env/cons/begin}{\zcsetup{countertype={thm=cons}}}
\theoremstyle{remark}
\AddToHook{env/rem/begin}{\zcsetup{countertype={thm=remark}}}
\newtheorem{rem}[thm]{Remark}

\newtheorem*{convention}{Convention}

\makeatletter
\newcommand{\WronglyDeclarePairedDelimiter}[3]{%
  \expandafter\DeclarePairedDelimiter\csname RIGHT\string#1\endcsname{#2}{#3}
  \newcommand#1{%
    \@ifstar{\csname RIGHT\string#1\endcsname}
            {\csname RIGHT\string#1\endcsname*}%
  }%
}
\makeatother
\WronglyDeclarePairedDelimiter{\abs}{\lvert}{\rvert}
\WronglyDeclarePairedDelimiter{\norm}{\lVert}{\rVert}

\newcommand{\df}{\textit}

\newcommand{\union}{\cup}
\newcommand{\inter}{\cap}

\newcommand{\mc}{\mathcal}

\newcommand{\mf}{\mathfrak}

\renewcommand{\H}{\mathbb{H}}

\newcommand{\R}{\mathbb{R}}
\renewcommand{\C}{\mathbb{C}}
\newcommand{\N}{\mathbb{N}}

\newcommand{\Q}{\mathbb{Q}}

\newcommand{\Sph}{\mathbb{S}}

\newcommand{\Id}{\mathsf{Id}}

\DeclareMathOperator{\PSL}{\mathsf{PSL}}

\DeclareMathOperator{\SL}{\mathsf{SL}}
\DeclareMathOperator{\GL}{\mathsf{GL}}

\DeclareMathOperator{\Aut}{Aut}
\DeclareMathOperator{\Isom}{Isom}
\DeclareMathOperator{\Stab}{Stab}

\DeclareMathOperator{\tr}{tr}

\DeclareMathOperator{\Hom}{Hom}

\DeclareMathOperator{\Hol}{Hol}

\begin{document}
\zcsetup{noabbrev, cap}

\title[Expansion joints in hyperbolic manifolds]{Expansion joints in hyperbolic manifolds}
\author[A. Elzenaar]{Alex Elzenaar}
\address{School of Mathematics, Monash University, Melbourne}
\email{alexander.elzenaar@monash.edu}
\thanks{The author was supported by an Australian Government Research Training Program (RTP) Scholarship during the period that this work was undertaken. This document was prepared without the use of any generative AI}


\begin{abstract}
  Deformations of hyperbolic manifolds through metrics with cone singularities along closed loops were first studied by Thurston as continuous realisations of Dehn
  fillings. Instead of gluing singular solid tori into rank $2$ cusps, we glue singular $2$-handles into rank $1$ cusps. To do this we find substructures
  within which the hyperbolic metric can be `fractured' in a controlled way by direct manipulation of a fundamental polyhedron, changing the cone angle around
  an ideal arc to interpolate between cusped hyperbolic manifolds and hyperbolic manifolds with conformal surfaces on the visual boundary.
  As an application, we use cone deformations of a family of arithmetic manifolds derived from the Borromean rings to show that the upper unknotting tunnels of highly
  twisted $2$-bridge links can be drilled out by cone deformations through pinched negatively curved metrics. Finally we show that our structures arise naturally
  in fully augmented links, providing a large family of examples.
\end{abstract}

\maketitle
\section{Introduction}
Hyperbolic cone manifolds (see \cite[\S3]{boileau05} for definitions) provide geometric structures which interpolate between complete hyperbolic metrics on different topological manifolds: hyperbolic cusps are
limiting cases of cone singularities where the cone angle is $0$, and arcs away from the singular locus are limiting cases where the cone angle is $2\pi$. The holonomy
group of a hyperbolic cone manifold is a subgroup of $ \PSL(2,\C) $ but is only discrete if the cone angles are submultiples of $ 2\pi $ (i.e.\ the cone manifold is an
orbifold). It was observed by Thurston~\cite[\S4.4--4.5]{thurstonN} (see also \cite{thurston98}) that a path of cone manifolds in which a cone angle around a closed loop is increased
from $ 0 $ to $2\pi$ is a continuous realisation of a Dehn filling, with slope depending on the choice of holonomy element which becomes elliptic. This idea has been used
by Brock and Bromberg~\cite{brock04,bromberg04} to prove a version of the density conjecture for Kleinian groups, and by Hodgson and Kerckhoff~\cite{hk05} and Futer, Purcell,
and Schleimer~\cite{futer22b,futer22} to compare the hyperbolic metric on a link complement to that on its Dehn fillings by showing that the metric on a family of interpolating cone manifolds is controlled.

We wish to study the geometry of finite volume hyperbolic $3$-manifolds like link complements by cone-deforming them into manifolds with nontrivial conformal boundary
in a controlled way. This is motivated by questions about geodesicity of certain embedded arcs in link complements
and by a more general need for effective tools that can compare finite and infinite covolume Kleinian groups. To model the
topological operation of gluing in a $2$-handle across a rank $1$ cusp with a path of cone manifolds, we need to be able to deform cone angles around an ideal arc
continuously through the interval $ [0,2\pi] $. An abstract framework to prove existence of cone deformations was a major programme of work in the early 21st century
building on ideas of Hodgson and Kerckhoff~\cite{hk98}, see also Wei\ss~\cite{weiss05,weiss07} and Kojima~\cite{kojima98} for the case that the singular locus
is a link and Wei\ss~\cite{weiss13} and Montcouquiol~\cite{montcouquiol13} for the case that the singular locus is a graph. Unfortunately, to apply the global results
of this theory in practice it is necessary to obtain detailed technical estimates such as injectivity radius bounds. One of the major difficulties
is that there may be some `critical angle' where the metric ceases to be hyperbolic (see e.g.\ \cite{hodgson,porti02,cooper18,heusener01,porti98}). Thus
to work with specific examples it is often still necessary to construct the deformations explicitly by hand, as in work of Akiyoshi~\cite{akiyoshi18}, Yoshida~\cite{yoshida22}, and the author~\cite{elzenaar24c}.

In this paper, we exhibit and study substructures in hyperbolic $3$-manifolds which isolate a family of rank $2$ cusps and allow arcs joining them to be cone-deformed
to form new rank $1$ cusps. In the process the metric around the cusps is fractured so that they bubble into punctured surfaces on the conformal boundary, while the remainder
of the manifold remains geometrically controlled. We call these structures \textit{expansion joints} since they behave like expansion joints on a bridge, allowing two halves
of the manifold to move apart and isolating all stress in a single location. Other geometric isolation phenomena were observed by Neumann and Reid~\cite{neumann91,neumann93},
see also the later work of Kapovich~\cite{kapovich92} and Calegari~\cite{calegari96,calegari01}; they gave examples of $3$-manifolds in which some cusps can be filled via a
cone deformation while preserving the Euclidean structure on other cusps. After doubling across the conformal boundary, our form of isolation is similar except that there is a
rank $2$ cusp which is transverse to the region containing cone deformations. Another approach to the study of cusp shape changes upon cone deformations is due to Purcell~\cite{purcell08},
who used the Hodgson--Kerckhoff theory to bound changes to the Euclidean metrics on cusp tori; our method gives stronger results when it applies, but requires control over
a local polyhedral decomposition.

\begin{figure}
  \centering
  \begin{subfigure}[c]{.24\textwidth}
    \centering
    \includegraphics[width=\textwidth]{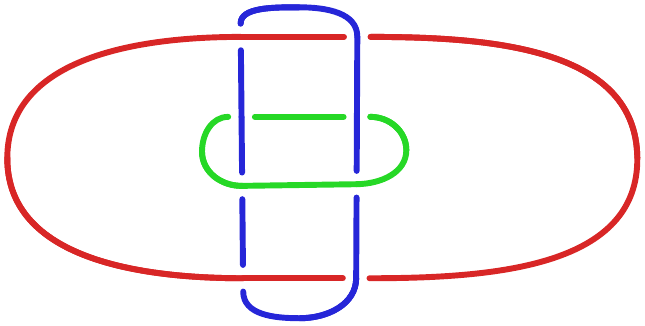}
    \caption{Borromean rings.\label{fig:borromean_rings}}
  \end{subfigure}\hspace{1em}
  \begin{subfigure}[c]{.24\textwidth}
    \centering
    \includegraphics[width=\textwidth]{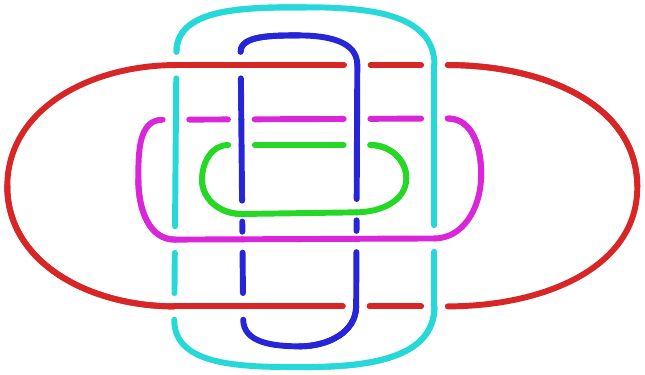}
    \caption{$n=2$.}
  \end{subfigure}\hspace{1em}
  \begin{subfigure}[c]{.24\textwidth}
    \centering
    \includegraphics[width=\textwidth]{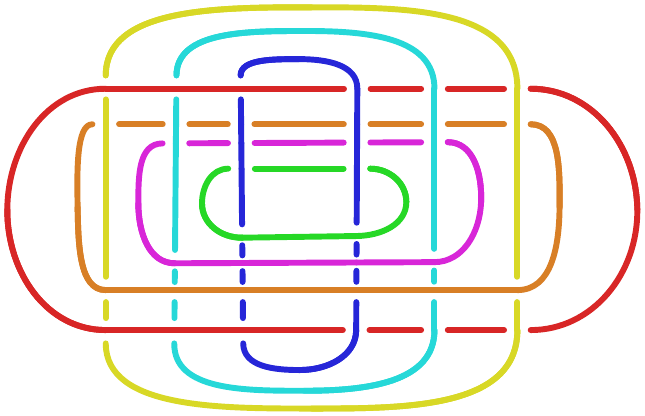}
    \caption{$n=3$.}
  \end{subfigure}
  \caption{The $n$-stacked Borromean rings for small $n$.\label{fig:stacked_rings}}
\end{figure}

\begin{figure}
  \centering
  \begin{subfigure}[t]{.42\textwidth}
    \centering
    \labellist
    \small\hair 2pt
    \pinlabel {$\ast$} [l] at 193 113
    \endlabellist
    \includegraphics[height=.1\textheight]{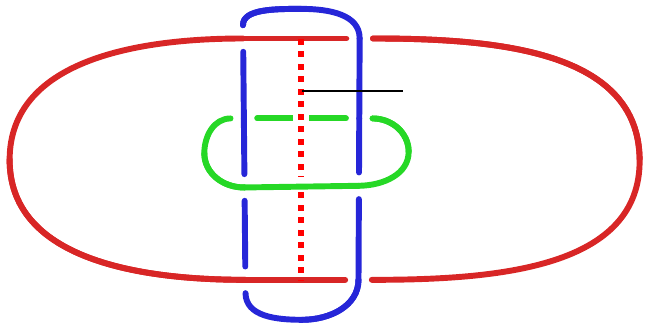}
    \caption{Cone arc in the link complement.\label{fig:borromean_rings_conearc}}
  \end{subfigure}\hspace{1em}
  \begin{subfigure}[t]{.48\textwidth}
    \centering
    \includegraphics[height=.1\textheight]{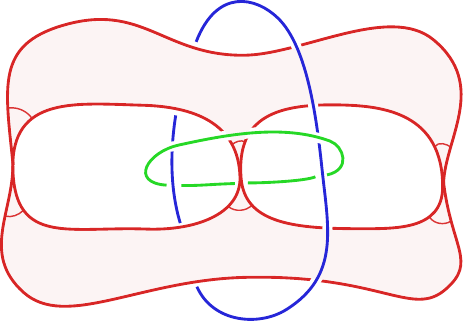}
    \caption{As the angle around $ \ast $ goes to $0$, the arc becomes a rank $1$ cusp.\label{fig:lanternfold_intro}}
  \end{subfigure}
  \caption{The lantern manifold is produced by deforming the cone angle in the Borromean rings complement around the dotted arc $\ast$ from $ 2\pi $ to $0$.}
\end{figure}

Our starting point is a family of arithmetic links called \df{stacked Borromean rings} (\zcref{defn:explicit}) obtained by choosing a distinguished `central'
component of the Borromean rings and repeating the other two components radially outwards $n$ times as shown in \zcref{fig:stacked_rings} to obtain $2n + 1$ loops in total.
The complement in $ \Sph^3 $ of the $n$-stacked Borromean rings, denoted $ B^n $, admits a complete hyperbolic structure.
For $ n = 1 $ (the usual Borromean rings) this is classical: a geometric proof may be found in Thurston~\cite[\S 3.4]{thurstonN}, and
an algebraic proof following Riley may be found in Wielenberg~\cite{wielenberg78}. For $ n > 1 $ this will be proved as \zcref{thm:octahedra}.
In \zcref{prp:cone_deform_simple}, we prove that there is a continuous path of hyperbolic cone
manifolds which interpolate between the Borromean rings complement and the \df{lantern manifold} shown in \zcref{fig:lanternfold_intro} (the name is chosen since the motivation for
its study came from the lantern configuration in mapping class theory~\cite[\S5.1.1]{farb}). The cone metrics on the path each have a single singular arc, labelled with $\ast$
in \zcref{fig:borromean_rings_conearc}, with cone angle ranging from $ 2\pi$ to $0$. Our first main result is a generalisation of this cone deformation to all
stacked Borromean rings:
\begin{mainthm}[\zcref{thm:cone_deform_complex}]\label{thm:main}
  There exists a smooth path $ p : [0, 2\pi] \to \Hom(F_{2n+2}, \PSL(2,\C)) $ so that for each $ \theta > 0 $ the image $ p(\theta) $ is the holonomy group of a
  hyperbolic cone manifold, supported on $B^n$, with a single singular arc indicated in \zcref{fig:intro_cone_arc} of angle $\theta$.
  When $ \theta = 0 $, $ p(\theta) $ is the holonomy group of the complete hyperbolic manifold obtained by drilling the singular arc from $ B^n$.
\end{mainthm}
\begin{figure}
  \begin{subfigure}[t]{.48\textwidth}
    \centering
    \labellist
    \small\hair 2pt
    \pinlabel {$\ast$} [l] at 271 164
    \pinlabel {$\dagger$} [l] at 271 184
    \endlabellist
    \includegraphics[width=.8\textwidth]{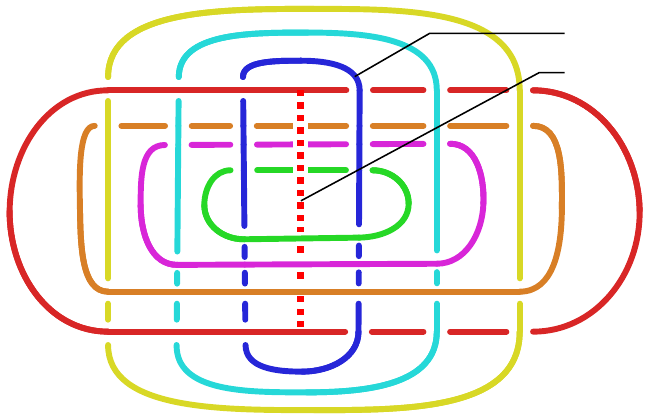}
    \caption{The manifold $ M^n $ is produced from $ B^n $ by decreasing the cone angle around the dotted arc $ \ast $ from $ 2\pi $ to $ 0 $.\label{fig:intro_cone_arc}}
  \end{subfigure}\hfill%
  \begin{subfigure}[t]{.48\textwidth}
    \centering
    \includegraphics[width=.8\textwidth]{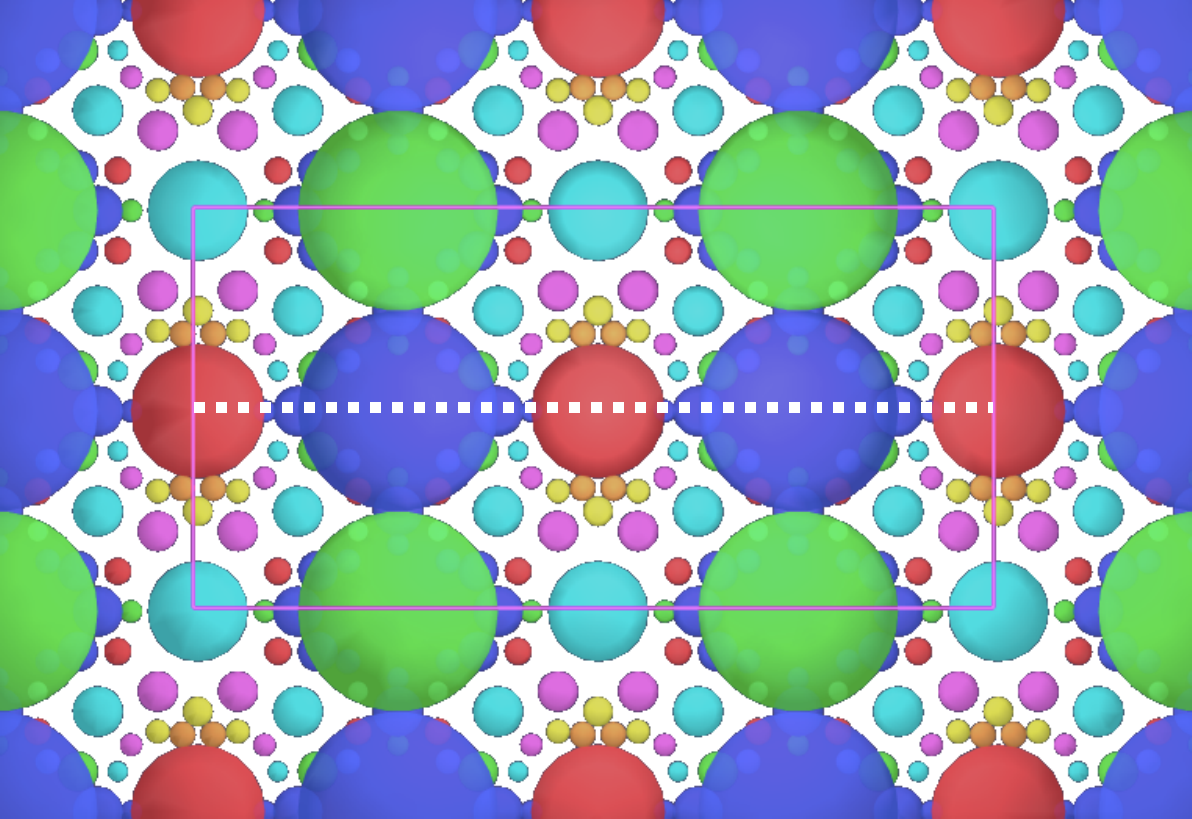}
    \caption{Cusp horoballs for $B^3$, with $\dagger$ at $ \infty $. The geodesic arc $ \ast $ lifts to a family of geodesic arcs above the dotted horizontal line.\label{fig:borromean_7cpt_cusps}}
  \end{subfigure}
  \caption{The combinatorics and geometry of $ B^3 $.}
\end{figure}

One indication that the manifolds $ B^n$ should admit a nice theory for their hyperbolic structure comes from plotting horoball neighbourhoods of their cusps
as shown for $n=3$ in \zcref{fig:borromean_7cpt_cusps}. Our horoball images were produced by \texttt{SnapPy}~\cite{SnapPy32}, and for detailed
information on how to interpret these figures see Thurston~\cite{thurston98}. An important theme in this paper is that the expansion joint cone-deformations which we
construct arise from structures that can be guessed from these figures.

\zcref[S]{thm:main} is proved by first using the symmetries visible in the cusp diagram to decompose $ B^n $ into ideal right-angled octahedra and then deforming the induced fundamental
polyhedron in $\H^3$. Decompositions of knot and link complements into right-angled octahedra are of independent interest in knot theory, due to their connections with
number theory and the geometry of circle packings, see e.g.\ \cite{ibarra25,pinsky23,purcell09}. As an application of \zcref{thm:main}, we prove in \zcref{cor:unknotting_tunnel}
that the upper unknotting tunnel of a highly twisted two-bridge knot can be cone-deformed from angle $ 2\pi $ to angle $0$ through pinched negatively curved metrics which have
constant sectional curvature $-1$ in a neighbourhood of the cone singularity; it is this application which originally motivated
the definition of stacked Borromean rings. This application generalises to all manifolds containing the sorts of expansion joints which we study. We also
provide a fast proof of a conjecture of Sakuma, that any the upper unknotting tunnel of any hyperbolic two-bridge link can be cone-deformed from angle $ 2\pi $ to angle $0$
through \textit{hyperbolic} metrics, based on an idea supplied by the referee in the context of lantern manifolds. This is done in \zcref{rem:sakuma_conjecture}, but we
believe this proof is probably known to some researchers already even though it does not seem to be in found anywhere in the literature.

The specific example of $ B^n$ is just a convenient setting for describing our methods that is complex enough to show most of the interesting features, and
in the final section of the paper we describe a more general kind of expansion joint:
\begin{mainthm}[\zcref{thm:concave_lens}]\label{thm:main2}
  There exists a polyhedral substructure $ \mc{L} $, which can be detected locally in a hyperbolic $3$-manifold and which deformation-retracts onto
  an embedded geodesic surface, that acts as an expansion joint. That is, its existence in a manifold $M$ allows the construction of a smooth path of cone manifold structures on $ M $
  that, in the limit, deforms a family of rank $2$ cusps into a family of embedded thrice-punctured spheres, without modifying the hyperbolic structure of $ M $ away from $ \mc{L} $.
\end{mainthm}
As an application of \zcref{thm:main2}, we show in \zcref{prp:fal} that all holonomy groups of fully augmented link complements can be obtained by taking an infinite covolume Kleinian
group and continuously deforming a parabolic element through infinite order elliptics until it becomes the identity in such a way that all intermediate groups
are holonomy groups of cone manifolds with no unexpected singularities. This exhibits every fully augmented link group as the starting point of a smooth arc embedded inside some character variety, parameterising
controlled indiscrete groups, that ends on the boundary of a nontrivial quasiconformal deformation space. Since every link in $ \Sph^3 $ arises as a Dehn filling of a fully augmented link,
by applying the same techniques as \zcref{cor:unknotting_tunnel} we obtain cone deformations for pinched negatively curved metrics on all hyperbolic links in $ \Sph^3 $ that join them to a pinched
metric on an infinite-volume-hyperbolisable $3$-manifold.

In order for the paper to be self-contained we include, in \zcref{sec:poincare}, a version of the Poincar\'e polyhedron theorem for cone manifolds which is well-known but to the best
of our knowledge does not appear explicitly in the literature. We also indicate how to write down Maskit combination theorems for cone manifold holonomy groups. This is an important tool in our
general programme to apply the theory of Kleinian groups to well-behaved indiscrete subgroups of $\PSL(2,\C) $ in order to join different islands of discreteness
in the character variety.

\begin{convention}
  All polyhedra are metric polyhedra (i.e.\ are locally modelled by intersections of half-spaces). Any exceptions will
  be called \df{combinatorial polyhedra}.
\end{convention}

\subsection*{Acknowledgments}
I thank Jessica Purcell for discussion surrounding this work. I also thank the anonymous referees for very helpful comments, particularly providing the central ideas
for \zcref{rem:hk_deformation}, \zcref{rem:sakuma_conjecture}, and \zcref{rem:hk_fal}. Many images of knot diagrams and cusp horoballs were produced by the \texttt{SnapPy} software~\cite{SnapPy32}.

\section{The classical Borromean rings}\label{sec:borromean}
We consider the topological $3$-manifold $M$ obtained by drilling two closed loops from a genus $2$ handlebody as in \zcref{fig:lanternfold}. The group
$ \pi_1(M) $ is generated by the four loops $ A, X, Y, Z $ shown in the figure. The manifold $ M $ has a complete hyperbolic structure if and only if there exists
a faithful discrete representation $ \rho : \pi_1(M) \to \PSL(2,\C) $ such that $ M \simeq_{\mathrm{homeo.}} \H^3/\rho(\pi_1(M)) $. We will construct
such a representation, with the additional condition that the loops $ Y $, $ Z $, and $ ZY^{-1} $ are homotopic to rank $1$ cusps on the conformal boundary of $M$, as in \zcref{fig:lanternfold_intro};
this forces $ \rho $ to be a maximal cusp representation on the boundary of the quasiconformal deformation space of all hyperbolic structures, and in particular
it will be rigid. Identifying the four generators of $ \pi_1(M) $ with their images in $ \PSL(2,\C) $, the representation must satisfy the conditions in \zcref{tab:equations}.
These conditions come from the usual relator for a genus $2$ surface group, together with commutation conditions for the torus boundary components;
we write $ [R,S] $ for the commutator $ RSR^{-1}S^{-1} $. It is not clear \textit{a priori} that these relations are sufficient to pin down the hyperbolic structure,
but we will show in \zcref{prp:borromean_variety} that they are.

\begin{figure}
  \centering
  \includegraphics[width=.7\textwidth]{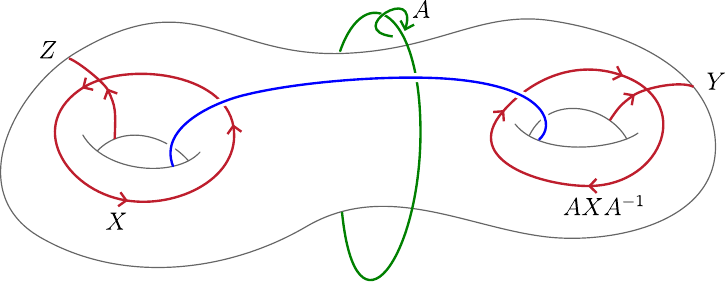}
  \caption{The lantern manifold $M$: a genus $2$ handlebody with two drilled loops, one green (with meridian $A$) and one blue (with meridian $X$). Here, and elsewhere in the paper, manifolds are drawn from the perspective of a viewer in the interior. \label{fig:lanternfold}}
\end{figure}

\begin{table}
  \centering
  \caption{Equations cutting out the character variety.\label{tab:equations}}
  \begin{tabular}{ll}\toprule
    \multicolumn{2}{l}{\textsl{Surface:}}\\
    Rank $1$ cusps & $ \tr^2 Z = \tr^2 Y = \tr^2 ZY^{-1} = 4 $\\
    Surface relation & $ [Z^{-1}, X] [Y^{-1}, AXA^{-1}]^{-1} = \Id $\\\midrule
    \multicolumn{2}{l}{\textsl{Vertical rank $2$ cusp:}}\\
    Meridian & $ \tr^2 A = 4 $\\
    Longitude & $ \tr^2\, [ Z^{-1}, X] = 4 $\\
    Abelian & $ [ A, [ Z^{-1}, X] ] = \Id $\\\midrule
    \multicolumn{2}{l}{\textsl{Horizontal rank $2$ cusp:}}\\
    Meridian & $ \tr^2 X = 4 $\\
    Longitude & $ \tr^2 ZA^{-1}Y^{-1}A = 4$\\
    Abelian & $ [X, ZA^{-1}Y^{-1}A] = \Id $\\\bottomrule
  \end{tabular}
\end{table}

We will convert these conditions into a set of polynomial equations in $ \GL(2, \C[\vec{x}]) $ where $ \vec{x} $ is some suitable list of indeterminates. Up to conjugacy, we may take
\begin{displaymath}
  A = \begin{bmatrix} 1 & 1 \\ 0 & 1 \end{bmatrix}\quad\text{and}\quad X = \begin{bmatrix} 1 & 0 \\ a_3 & 1 \end{bmatrix}.
\end{displaymath}
This takes care of $ \tr^2 A = 4 $ and $ \tr^2 X = 4 $. (It is not \textit{a priori} clear that we should be able to choose
a global lift from $\PSL$ to $\SL$ where $ \tr A $ and $ \tr X $ have equal sign, but the assumption is justified \textit{a posteriori}
as it leads to a solution.) We now let
\begin{displaymath}
  Y = \begin{bmatrix} b_1 & b_2 \\ b_3 & b_4 \end{bmatrix}\quad\text{and}\quad Z = \begin{bmatrix} c_1 & c_2 \\ c_3 & c_4 \end{bmatrix}
\end{displaymath}
where the $ b_i $ and $ c_i $ are complex indeterminates. We impose on these the determinant conditions $ \det Y = \det Z = 1 $, the remaining trace conditions in the table,
and the surface relation $ [Z^{-1}, X] = \pm [Y^{-1}, AXA^{-1}] $. Since we chose $ A $ and $ X $ to have fixed points at $ \infty $ and $0$ respectively
and parabolics commute if and only if they share fixed points, $[ A, [ Z^{-1}, X] ] = \Id$ if and only if $ [ Z^{-1}, X]_{(2,1)} = 0 $ (here, the subscript denotes
matrix indexing; i.e.\ we are requiring the lower-left entry of the matrix $ [Z^{-1},X] $ to be $0$) and $ [X, ZA^{-1}Y^{-1}A] = \Id $ if and only if $ (ZA^{-1}Y^{-1}A)_{(1,2)} = 0 $.
All together we obtain an \textit{a priori} overdetermined polynomial system, which has a $0$-dimensional solution set:

\begin{prp}\label{prp:borromean_variety}
  There are exactly $16$ solutions in $ \C^9 $ to the polynomial system just described which satisfy the non-degeneracy conditions
  \begin{displaymath}
    a_3 \neq 0,\quad [ Z^{-1}, X]_{(1,2)} \neq 0,\quad\text{and}\quad (ZA^{-1}Y^{-1}A)_{(2,1)} \neq 0,
  \end{displaymath}
  which prevent various parabolics from degenerating to the identity. Eight of these give groups conjugate after change of generators to the group $G_{2\pi} $ with parameters
  \begin{displaymath}
    (a_3,b_1,b_2,b_3,b_4,c_1,c_2,c_3,c_4) = (-2i, -i, 1, -2i, 2 + i, i, -1, 2i, -2 - i).
  \end{displaymath}
  The remaining eight give groups conjugate after change of generators to the group $ G_0 $ with parameters
  \begin{displaymath}
    (a_3,b_1,b_2,b_3,b_4,c_1,c_2,c_3,c_4) = \left(-4i, -3i, \frac{3}{2} + 2i, -4i, 2 + 3i, i, -\frac{1}{2}, 4i, -2 - i\right).
  \end{displaymath}
\end{prp}
\begin{proof}
  The system of equations may be solved exactly by computer algebra systems (e.g.\ using the \texttt{Reduce} command in \texttt{Mathematica}) to verify
  that there are $16$ solutions; the conjugacies that can be defined to separate these into the two sets of eight are essentially immediate from the symmetries
  of the solutions.
\end{proof}

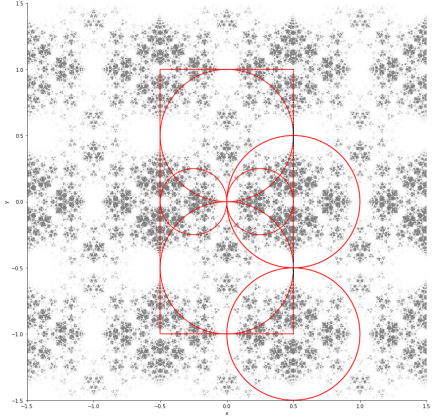
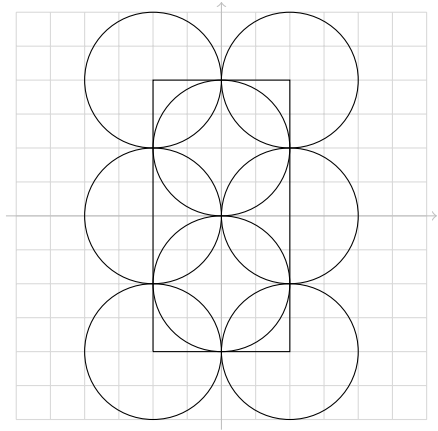
\begin{figure}
  \begin{subfigure}[t]{0.45\textwidth}
    \centering
    \includegraphics[width=\textwidth]{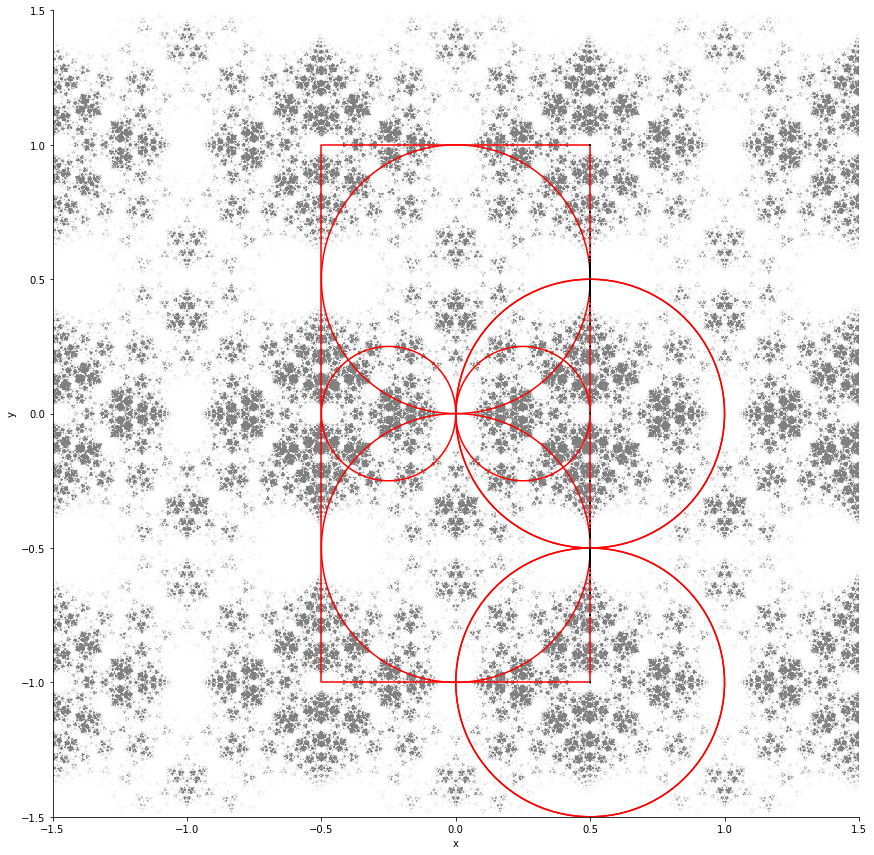}
    \caption{Limit set and isometric circles of $X$ and $ ZA^{-1}Y^{-1}A $ (the respective meridian and longitude of the cusp descending from $0$), and $Y=Z$ (the meridian of the cusp coloured red in \zcref{fig:borromean_rings}).\label{fig:limset_borromean}}
  \end{subfigure}\hfill
  \begin{subfigure}[t]{0.45\textwidth}
    \centering
    \begin{tikzpicture}[scale=1.8]
      \draw [thin, gray!30, step=.25cm,xshift=-1cm, yshift=-0.5cm] (-.5,-1) grid +(3,3);
      \draw [thin, gray!50, ->] (0,-1.575) -- (0,1.575);
      \draw [thin, gray!50, ->] (-1.575,0) -- (1.575,0);
      \draw[fill=none] (0,.5) circle (.5);
      \draw[fill=none] (0,-.5) circle (.5);
      \draw[fill=none] (.5,1) circle (.5);
      \draw[fill=none] (.5,0) circle (.5);
      \draw[fill=none] (.5,-1) circle (.5);
      \draw[fill=none] (-.5,1) circle (.5);
      \draw[fill=none] (-.5,0) circle (.5);
      \draw[fill=none] (-.5,-1) circle (.5);
      \draw (-.5,-1)--(-.5,1)--(.5,1)--(.5,-1)--(-.5,-1);
    \end{tikzpicture}
    \caption{The solid circles and lines give the projection of a fundamental domain onto the Riemann sphere. Grid squares have width $ 0.25 $.\label{fig:fd_borromean}}
  \end{subfigure}
  \caption{Data associated with the Borromean rings group $ G_{2\pi} $.\label{fig:borromean}}
\end{figure}

\begin{ex}[The Borromean rings]
  The group $ G_{2\pi} $ admits the relation $ ZY^{-1} = \Id $. In fact it is essentially the same realisation of the Borromean rings group
  as that constructed by Wielenberg~\cite[Example~8]{wielenberg78}. We show the limit set and fundamental domain in \zcref{fig:borromean}, which should be compared with Figure~12 \textit{op. cit}.
\end{ex}

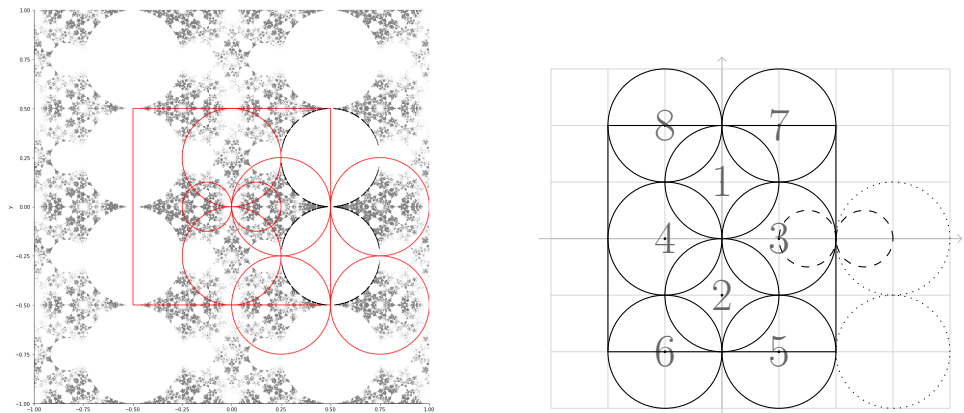
\begin{figure}
  \begin{subfigure}[t]{0.45\textwidth}
    \includegraphics[width=\textwidth]{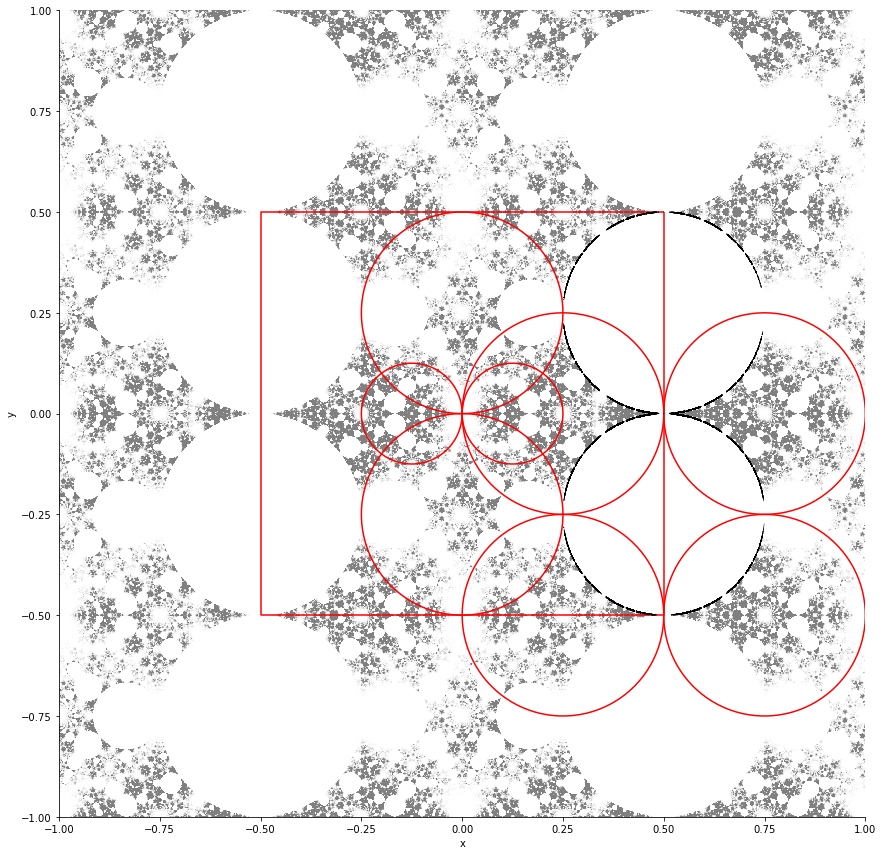}
    \caption{Limit set and isometric circles of $X$ and $ ZA^{-1}Y^{-1}A $ (the respective meridian and longitude of the cusp descending from $0$), and $Y$ and $Z$ (which generate the stabiliser of one of the thrice-punctured sphere components).\label{fig:limset_lantern}}
  \end{subfigure}\hfill
  \begin{subfigure}[t]{0.45\textwidth}
    \centering
    \begin{tikzpicture}[scale=3]
      \draw [thin, gray!30, step=.25cm,xshift=-0.5cm, yshift=-0.5cm] (-.25,-.25) grid +(1.75,1.5);
      \draw [thin, gray!50, ->] (0,-.805) -- (0,.805);
      \draw [thin, gray!50, ->] (-.805,0) -- (1.055,0);
      \draw[fill=none] (0,.25) circle (.25);
      \draw[black,fill=black](0,.25) circle (.1pt) node {\huge\color{black!60} $1$};
      \draw[fill=none] (0,-.25) circle (.25);
      \draw[black,fill=black](0,-.25) circle (.1pt) node {\huge\color{black!60} $2$};
      \draw[fill=none] (.25,0) circle (.25);
      \draw[black,fill=black](.25,0) circle (.1pt) node {\huge\color{black!60} $3$};
      \draw[fill=none] (-.25,0) circle (.25);
      \draw[black,fill=black](-.25,0) circle (.1pt) node {\huge\color{black!60} $4$};
      \draw[fill=none] (.25,-.5) circle (.25);
      \draw[black,fill=black](.25,-.5) circle (.1pt) node {\huge\color{black!60} $5$};
      \draw[fill=none] (-.25,-.5) circle (.25);
      \draw[black,fill=black](-.25,-.5) circle (.1pt) node {\huge\color{black!60} $6$};
      \draw[fill=none] (.25,.5) circle (.25);
      \draw[black,fill=black](.25,.5) circle (.1pt) node {\huge\color{black!60} $7$};
      \draw[fill=none] (-.25,.5) circle (.25);
      \draw[black,fill=black](-.25,.5) circle (.1pt) node  {\huge\color{black!60} $8$};
      \draw (-.5,-.5)--(-.5,.5)--(.5,.5)--(.5,-.5)--(-.5,-.5);
      \draw[fill=none, dotted] (.75,0) circle (.25);
      \draw[fill=none, dotted] (.75,-.5) circle (.25);
      \draw[fill=none, dashed] (.375,0) circle (.125);
      \draw[fill=none, dashed] (.625,0) circle (.125);
    \end{tikzpicture}
    \caption{The solid circles and lines give the projection of a fundamental domain onto the Riemann sphere. The two dotted circles are paired by $ Y $, and the two
             dashed circles are paired by $ ZY^{-1}$. Grid squares have width $ 0.25 $.\label{fig:fd_lantern}}
  \end{subfigure}
  \caption{Data associated with the lantern group $ G_0 $.}
\end{figure}
\begin{table}
  \centering
  \caption{Side-pairing transformations for \zcref{fig:fd_lantern}.\label{tab:pairings}}
  \begin{tabular}{llcll}\toprule
    $1\leftrightarrow 2$ & $X$ &\quad& \\
    $3\leftrightarrow 5$ & $Z$ && $ 4\leftrightarrow 6$ & $ A^{-1} Y A $\\
    $7\leftrightarrow 3$ & $XZX^{-1} $ && $ 8\leftrightarrow 4$ & $ X A^{-1} Y A X^{-1} $\\\midrule
    \multicolumn{5}{c}{\textsl{Translations:} $A$ (horiz.), $ Z^{-1} XZ X^{-1} $ (vert.)}\\\bottomrule
  \end{tabular}
\end{table}

\begin{ex}[The lantern group]
  We claim that $ G_0 $ is the holonomy group of a complete hyperbolic structure on the manifold $M$. Plotting
  the limit set and isometric circles of the various words that appear in \zcref{tab:equations} we obtain \zcref{fig:limset_lantern},
  which suggests the construction of the fundamental domain given in \zcref{fig:fd_lantern} (the four corners of the square
  are at $ \pm \frac{1}{2} \pm \frac{1}{2} i $ and all circles that appear orthogonal or tangent are so). The side-pairings are given in \zcref{tab:pairings}.
  From \zcref{fig:fd_lantern} it is possible to immediately read off an ideal triangulation for the convex core of $ M $.
\end{ex}

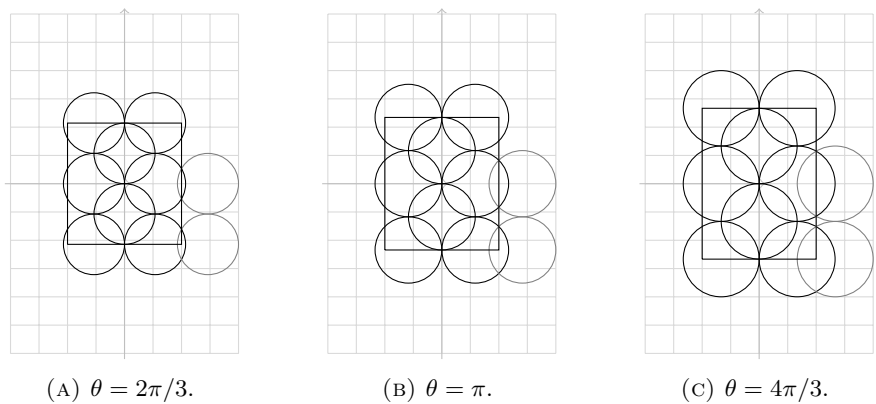
\begin{figure}
  \begin{subfigure}[c]{.33\textwidth}
    \centering
    \begin{tikzpicture}[scale=1.5]
      \draw [thin, gray!30, step=.25cm,xshift=-0.5cm, yshift=-0.5cm] (-.5,-1) grid +(2,3);
      \draw [thin, gray!50, ->] (0,-1.55) -- (0,1.55);
      \draw [thin, gray!50, ->] (-1.05,0) -- (.051,0);
      \draw[fill=none] (0, .2679) circle (.2679);
      \draw[fill=none] (0,-.2679) circle (.2679);
      \draw[fill=none] ( .2679,0) circle (.2679);
      \draw[fill=none] (-.2679,0) circle (.2679);
      \draw[fill=none] (.2679, -.5359) circle (.2679);
      \draw[fill=none] (-.2679,-.5359) circle (.2679);
      \draw[fill=none] ( .2679, .5359) circle (.2679);
      \draw[fill=none] (-.2679, .5359) circle (.2679);
      \draw[fill=none,gray] (.7321,  0) circle (.2679);
      \draw[fill=none,gray] (.7321,-.5359) circle (.2679);
      \draw (-.5,-.5359)--(-.5,.5359)--(.5,.5359)--(.5,-.5359)--(-.5,-.5359);
    \end{tikzpicture}
    \caption{$ \theta = 2\pi/3 $.}
  \end{subfigure}%
  \begin{subfigure}[c]{.33\textwidth}
    \centering
    \begin{tikzpicture}[scale=1.5]
      \draw [thin, gray!30, step=.25cm,xshift=-0.5cm, yshift=-0.5cm] (-.5,-1) grid +(2,3);
      \draw [thin, gray!50, ->] (0,-1.55) -- (0,1.55);
      \draw [thin, gray!50, ->] (-1.05,0) -- (.051,0);
      \draw[fill=none] (0, .2929) circle (.2929);
      \draw[fill=none] (0,-.2929) circle (.2929);
      \draw[fill=none] ( .2929,0) circle (.2929);
      \draw[fill=none] (-.2929,0) circle (.2929);
      \draw[fill=none] (.2929, -.5858) circle (.2929);
      \draw[fill=none] (-.2929,-.5858) circle (.2929);
      \draw[fill=none] ( .2929, .5858) circle (.2929);
      \draw[fill=none] (-.2929, .5858) circle (.2929);
      \draw[fill=none,gray] (.7071,  0) circle (.2929);
      \draw[fill=none,gray] (.7071,-.5858) circle (.2929);
      \draw (-.5,-.5858)--(-.5,.5858)--(.5,.5858)--(.5,-.5858)--(-.5,-.5858);
    \end{tikzpicture}
    \caption{$ \theta = \pi $.}
  \end{subfigure}%
  \begin{subfigure}[c]{.33\textwidth}
    \centering
    \begin{tikzpicture}[scale=1.5]
      \draw [thin, gray!30, step=.25cm,xshift=-0.5cm, yshift=-0.5cm] (-.5,-1) grid +(2,3);
      \draw [thin, gray!50, ->] (0,-1.55) -- (0,1.55);
      \draw [thin, gray!50, ->] (-1.05,0) -- (.051,0);
      \draw[fill=none] (0, .333) circle (.333);
      \draw[fill=none] (0,-.333) circle (.333);
      \draw[fill=none] ( .333,0) circle (.333);
      \draw[fill=none] (-.333,0) circle (.333);
      \draw[fill=none] (.333, -.6667) circle (.333);
      \draw[fill=none] (-.333,-.6667) circle (.333);
      \draw[fill=none] ( .333, .6667) circle (.333);
      \draw[fill=none] (-.333, .6667) circle (.333);
      \draw[fill=none,gray] (.667,  0) circle (.333);
      \draw[fill=none,gray] (.667,-.6667) circle (.333);
      \draw (-.5,-.6667)--(-.5,.6667)--(.5,.6667)--(.5,-.6667)--(-.5,-.6667);
    \end{tikzpicture}
    \caption{$ \theta = 4\pi/3 $.}
  \end{subfigure}
  \caption{Cone-deforming the Borromean rings to the lantern manifold. Each figure shows the fundamental domain for some $ G_{\theta} $.\label{fig:fd_coned}}
\end{figure}

\begin{prp}\label{prp:cone_deform_simple}
  There exists a smooth family of cone manifold holonomy groups $ G_{\theta} $ interpolating between $ G_0 $ and $ G_{2\pi} $; the manifold uniformised
  by $ G_\theta $ for $ \theta \in (0,2\pi] $ is supported on the complement of the Borromean rings, and has an ideal singular arc with cone angle $ \theta $
  along the arc labelled $\ast$ in \zcref{fig:borromean_rings_conearc}; as $ \theta \to 0 $, the manifolds limit onto the lantern manifold and $ G_\theta \to G_0 $.
\end{prp}
\begin{proof}
  The deformation is defined by continuously moving the walls of the fundamental domain while preserving its combinatorics and geometry away
  from the faces paired by $ ZY^{-1} $. From a comparison of \zcref{fig:fd_lantern} with \zcref{fig:fd_borromean}, a na\"ive description of
  the deformation is that we smoothly change the aspect ratio of the rank two lattice fixing $ \infty $ from $ 1:2 $ to $ 1:1 $, as in \zcref{fig:fd_coned}.

  To give a formal definition of the polyhedra used to define $ G_\theta $, observe that as the circles paired by $ Y $ and the circles paired
  by $ Z $ begin to overlap as the aspect ratio of the rectangle is changed, the angle between them is twice the angle between the two isometric
  circles of $ ZY^{-1} $. Thus, for the full deformation, as the angle between the two isometric circles of $ ZY^{-1} $ is deformed from $ 0 $
  (at $G_0$) to $ 2\pi $ (at $ G_{2\pi} $), the angle between the isometric circle of $ Z $ and the isometric circle of $ Y $ deforms from $ 0 $ to $ \pi $.
  Since we have normalised $ A $ to be the translation $ z \mapsto z+1 $, the width of the rectangle is fixed and so the height must vary to change the aspect
  ratio. The radius of the isometric circles of $ Y $ and $ Z $ is $1/4$ of the height of this rectangle, and an elementary trigonometric calculation
  shows that if the angle between the isometric circle of $Y$ and the isometric circle of $Z$ is $ 2\theta $, then the radii of those circles is
  $ r_\theta = \frac{1}{4} \sec^2 \frac{\theta}{2} $.

  This quantity is enough to write down a definition for the polyhedron that is used to produce $ G_{\theta} $.
  When $ \theta < 2\pi $ we take six circles of radius $ r_\theta $ with centres $ \epsilon r + \delta 2r i $ where $ \epsilon \in \{\pm 1\} $, $ \delta \in \{-1,0,1\} $,
  and two circles of radius $ r_\theta $ centred at $ \pm r_\theta i $. These, together with the four lines $ \pm 1/2 + \R i $ and $ \R \pm 2ri $ define
  a hyperbolic polyhedron with $14$ hyperbolic faces and four ideal faces on the Riemann sphere. It has two edges of dihedral angle $ \pi $, namely the intersections of domes with
  the hyperbolic plane lying above the real axis, and all other edges have dihedral angle $ \pi/2 $ except for those arising as intersections
  of domes with the two planes above the vertical lines $ \pm 1/2 + \R i $. There are four of these latter type of edges which each have dihedral angle $ \theta/4 $. When $ \theta \to 2\pi $
  this polyhedron degenerates to the polyhedron for the Borromean rings given in \zcref{fig:fd_borromean} and when $ \theta \to 0 $ it degenerates to the polyhedron
  given in \zcref{fig:fd_lantern}. Since it is combinatorially identical to the latter, we define side-pairings by writing down the uniquely defined M\"obius transformations
  which pair the faces in the same combinatorial pattern as in that group (i.e.\ the pattern of \zcref{tab:pairings}).

  The result then follows from a cone manifold version of the Poincar\'e polyhedron theorem, for instance the version described in \zcref{sec:poincare}, as an exercise in checking
  that all edge sums add to $ 2\pi $, except for a single edge cycle which has angle sum $ \theta $. By our choice of the embedding of the polyhedron in $ \H^3 $
  (e.g.\ choosing the horizontal width of the rectangle to be $1$) these transformations will obey the normalisations that we imposed earlier and so will interpolate
  directly between $ G_0 $ and $ G_{2\pi} $.
\end{proof}

\section{Stacked Borromean rings and lantern manifolds}\label{sec:lantern}

\begin{figure}
  \centering
  \labellist
  \small\hair 2pt
  \pinlabel {rank one cusps} [tl] at 23 13
  \pinlabel {$\mu_1$} [l] at 131 90
  \pinlabel {$\mu_n$} [l] at 131 76
  \pinlabel {$\alpha_1$} [r] at 186 90
  \pinlabel {$\alpha_n$} [r] at 186 76
  \endlabellist
  \includegraphics[width=.7\textwidth]{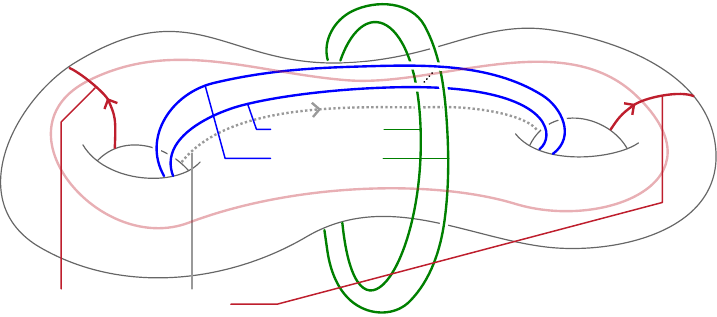}
  \caption{The $n$-stacked lantern manifold $M^n$, with drilled arcs of meridian $ \mu_i $ and $ \alpha_i $ alternating in a stack of crisscrossing drilled arcs. The central cusp $ \chi $ has become a closed loop on the genus $2$ surface, dual to two rank $1$ cusps.\label{fig:lanternfold_stack}}
\end{figure}

\begin{figure}
  \centering
  \labellist
  \small\hair 2pt
  \pinlabel {$\alpha_1$} [l] at 147 122
  \pinlabel {$\alpha_n$} [l] at 147 102
  \pinlabel {$\mu_1$} [l] at 147 51
  \pinlabel {$\mu_n$} [l] at 147 25
  \pinlabel {$\chi$} [r] at 0 111
  \endlabellist
  \includegraphics[width=.25\textwidth]{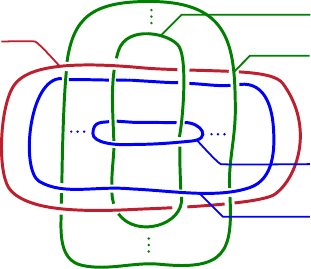}
  \caption{The components of the stacked Borromean rings, $ \mf{b}^n $.\label{fig:octahedra_link}}
\end{figure}
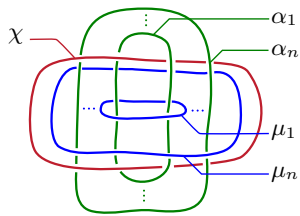

As decribed in the introduction, for our applications to $2$-bridge links we will be interested in manifolds where the two rank $2$ cusps in the lantern manifold are replaced
with $2n$ cusps (for some $ n \in \N $) that interleave as in \zcref{fig:lanternfold_stack}. We call this the \df{$ n$-stacked lantern manifold},
denoted $ M^n $. Gluing a $2$-handle along the dotted curve in the figure produces the complement $B^n$ of the $ n$-stacked Borromean rings $ \mf{b}^n $
For completeness, we give a formal definition of $M^n $ and $ \mf{b}^n $:
\begin{defn}\label{defn:explicit}
  The \df{$n$-stacked Borromean rings} is the link $ \mf{b}^n \subset \Sph^3 $ with a diagram in $ \R^2 $ produced by the following algorithm; compare with \zcref{fig:octahedra_link}.
  \begin{enumerate}
    \item Draw rectangles $ \mu_1, \ldots, \mu_n $, where $ \mu_i $ is determined by the diagonally-opposite corners $ ( -i, -i + 1/2 ) $ and $ (i, i - 1/2) $.
    \item Draw rectangles $ \alpha_1, \ldots, \alpha_n $, where $ \alpha_i $ is determined by the diagonally-opposite corners $ (-i + 1/2, -n - i/2) $ and $ (i - 1/2, n + i/2) $.
    \item At all crossing points $ (x,y) $ between any $ \mu_i $ and $ \alpha_j $, if $ y < 0 $ then define $ \mu_i $ to cross over $ \alpha_j $ and if $ y > 0 $ then
          define $ \alpha_j $ to cross over $ \mu_i $.
    \item Draw a rectangle $ \chi $ determined by the diagonally-opposite corners $ (-n- 1/2, -n) $ and $ (n+ 1/2, n) $.
    \item At all crossing points $ (x,y) $ between $ \chi $ and some $ \alpha_j $, if $ x < 0 $ then $ \chi $ crosses over $ \alpha_j $ and if $ x > 0 $ then $ \alpha_j $ crosses over $ \chi $.
  \end{enumerate}
  The \df{central component} is the component represented by the loop $ \chi $. The \df{$n$-stacked lantern manifold} is the topological manifold $M^n$ obtained by adding the arc which we will now describe to
  the diagram of $ \mf{b}^n $, lifting the result into $ \Sph^3 $ as an embedded graph, and taking the complement of this graph.
  \begin{enumerate}[resume]
    \item Draw a vertical segment $ \tau $ between $ (0, -n) $ and $ (0, n) $. This segment meets $ \chi $ at its two endpoints, which are trivalent vertices of the graph. All other crossing
          points $ (x,y) $ between $ \tau $ and a component of the diagram are incidences between $ \tau $ and some $ \mu_j $; if $ y > 0 $ then $ \tau $ crosses over the strand $ \mu_j $,
          and if $ y < 0 $ then $ \tau $ crosses under $ \mu_j $.
  \end{enumerate}
  See \zcref{fig:intro_cone_arc} in the introduction for an example of such a diagram for $ n = 3 $.
\end{defn}

\subsection{Hyperbolic structure}
Just as for the lantern group, we could set up and solve a system of polynomial equations in the entries of generators for the image of $ \pi_1(M^n) $ in $ \PSL(2,\C)$.
However, the number of variables grows with $ n $ and there is not much geometric enlightenment to be gained from the algebra. Instead, we construct the hyperbolic
structures by triangulating $B^n$ and then carrying out a cone deformation similar to that of \zcref{prp:cone_deform_simple}
to construct a fundamental domain for $M^n$.

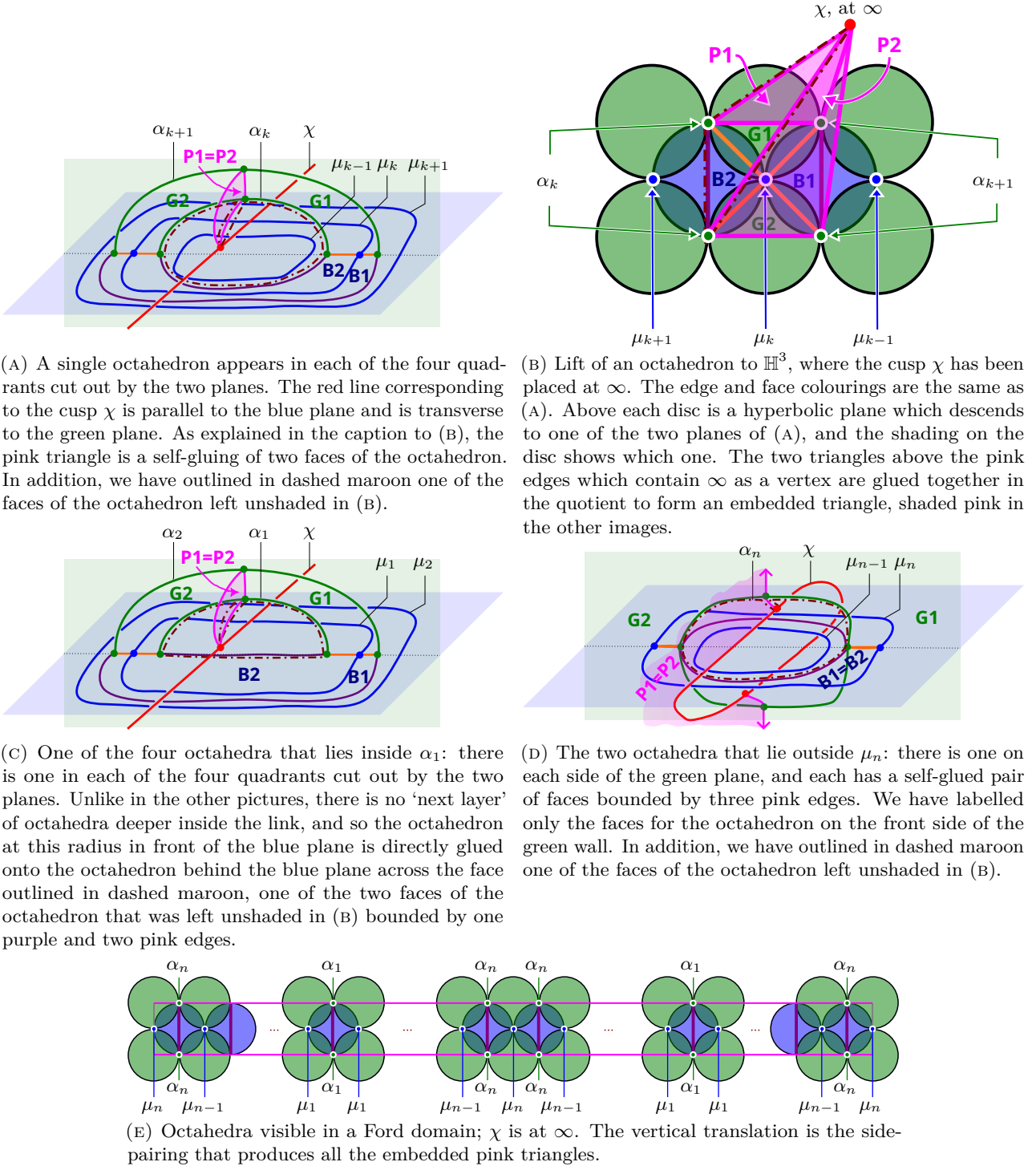
\begin{figure}
  \centerline{\begin{subfigure}[t]{.65\textwidth}
    \labellist
    \small\hair 2pt
    \pinlabel {$\mu_{k+1}$} [b] at 468 170
    \pinlabel {$\mu_{k}$} [b] at 423 170
    \pinlabel {$\mu_{k-1}$} [b] at 384 170
    \pinlabel {$\chi$} [b] at 337 207
    \pinlabel {$\alpha_k$} [b] at 283 207
    \pinlabel {$\alpha_{k+1}$} [b] at 188 207
    \endlabellist
    \includegraphics[width=\textwidth]{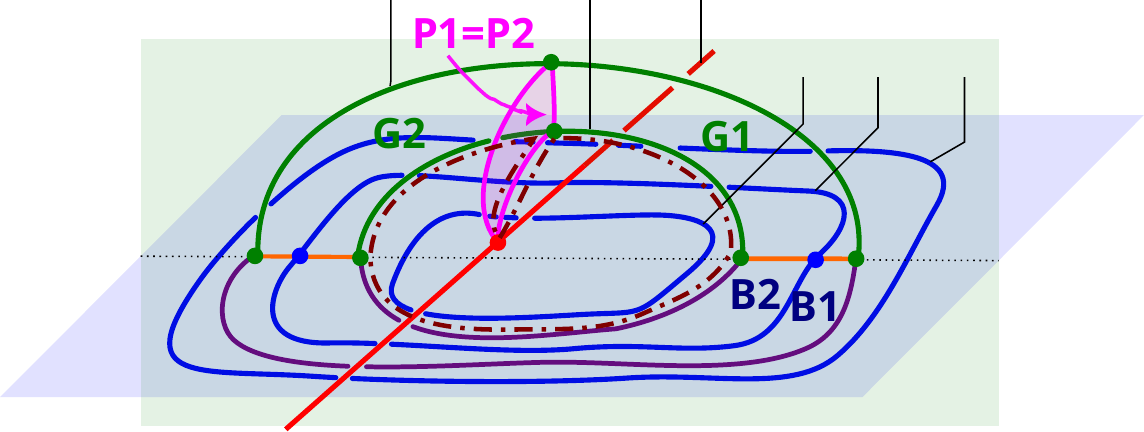}\vspace{.5em}
    \caption{A single octahedron appears in each of the four quadrants cut out by the two planes. The red line corresponding to the cusp $ \chi $ is parallel to the blue plane
            and is transverse to the green plane. As explained in the caption to (\textsc{b}), the pink triangle is a self-gluing of two faces of the octahedron. In addition,
            we have outlined in dashed maroon one of the faces of the octahedron left unshaded in (\textsc{b}).\label{fig:octahedra_s04s}}
  \end{subfigure}\hspace{.2cm}
  \begin{subfigure}[t]{.65\textwidth}\centering
    \labellist
    \small\hair 2pt
    \pinlabel {$\alpha_k$} at 0 58
    \pinlabel {$\alpha_{k+1}$} at 175 58
    \pinlabel {$\mu_{k+1}$} [t] at 41 0
    \pinlabel {$\mu_{k}$} [t] at 85 0
    \pinlabel {$\mu_{k-1}$} [t] at 128 0
    \pinlabel {$\chi$, at $\infty$} [b] at 118 122
    \endlabellist
    \includegraphics[width=.9\textwidth]{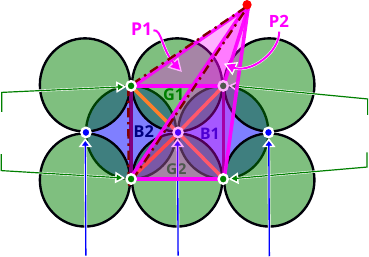}\vspace{.5em}
    \caption{Lift of an octahedron to $\H^3$, where the cusp $ \chi $ has been placed at $\infty$. The edge and face colourings are the same as (\textsc{a}). Above each disc is a hyperbolic plane
             which descends to one of the two planes of (\textsc{a}), and the shading on the disc shows which one. The two triangles above the pink edges which contain $\infty$ as a vertex are glued together in the quotient to form an embedded triangle, shaded pink in the other images.\label{fig:octahedra_single}}
  \end{subfigure}}
  \centerline{\begin{subfigure}[t]{.65\textwidth}
    \labellist
    \small\hair 2pt
    \pinlabel {$\alpha_2$} [b] at 187 206
    \pinlabel {$\alpha_1$} [b] at 282 206
    \pinlabel {$\chi$} [b] at 337 206
    \pinlabel {$\mu_1$} [b] at 421 172
    \pinlabel {$\mu_2$} [b] at 462 172
    \endlabellist
    \includegraphics[width=\textwidth]{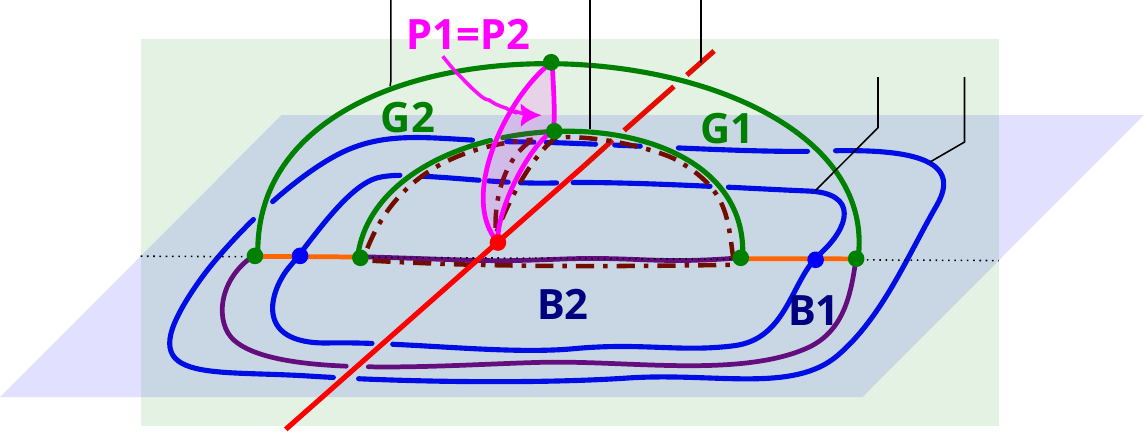}
    \caption{One of the four octahedra that lies inside $ \alpha_1 $: there is one in each of the four quadrants cut out by the two planes. Unlike in the other pictures, there is no `next layer' of octahedra
             deeper inside the link, and so the octahedron at this radius in front of the blue plane is directly glued onto the octahedron behind the blue plane across the face outlined in dashed maroon,
            one of the two faces of the octahedron that was left unshaded in (\textsc{b}) bounded by one purple and two pink edges.\label{fig:octahedra_s03s}}
  \end{subfigure}\hspace{.2cm}
  \begin{subfigure}[t]{.65\textwidth}
    \labellist
    \small\hair 2pt
    \pinlabel {$\alpha_n$} [b] at 250 186
    \pinlabel {$\chi$} [b] at 314 186
    \pinlabel {$\mu_{n-1}$} [b] at 378 172
    \pinlabel {$\mu_n$} [b] at 420 172
    \endlabellist
    \includegraphics[width=\textwidth]{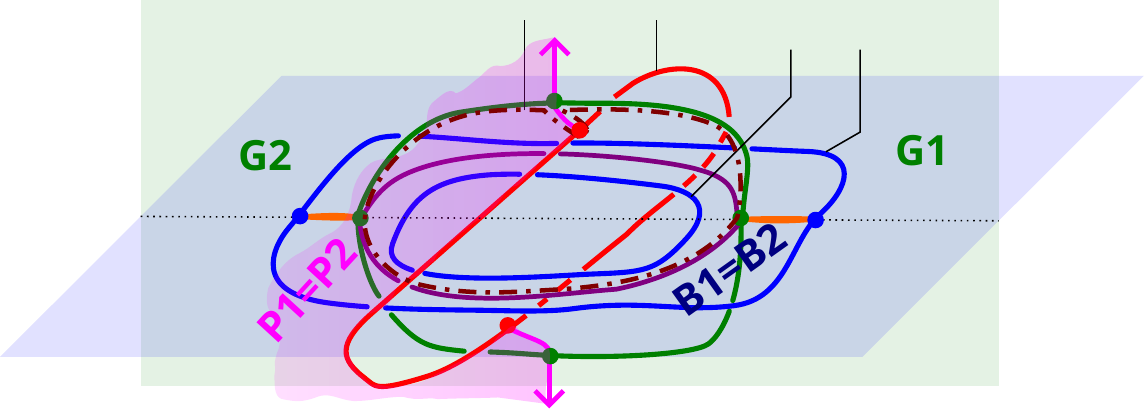}
    \caption{The two octahedra that lie outside $ \mu_n $: there is one on each side of the green plane, and each has a self-glued
    pair of faces bounded by three pink edges. We have labelled only the faces for the octahedron on the front side of the green wall. In addition, we have outlined in dashed maroon
    one of the faces of the octahedron left unshaded in (\textsc{b}).\label{fig:octahedra_s03s_outer}}
  \end{subfigure}}\vspace{1em}
  \begin{subfigure}{\textwidth}
    \labellist
    \small\hair 2pt
    \pinlabel {$\mu_n$} [t] at 18 0
    \pinlabel {$\mu_{n-1}$} [t] at 54 0
    \pinlabel {$\mu_{1}$} [t] at 128 0
    \pinlabel {$\mu_{1}$} [t] at 164 0
    \pinlabel {$\mu_{n-1}$} [t] at 238 0
    \pinlabel {$\mu_{n}$} [t] at 275 0
    \pinlabel {$\mu_{n-1}$} [t] at 310 0
    \pinlabel {$\mu_{1}$} [t] at 384 0
    \pinlabel {$\mu_{1}$} [t] at 422 0
    \pinlabel {$\mu_{n-1}$} [t] at 494 0
    \pinlabel {$\mu_{n}$} [t] at 530 0
    \pinlabel {$\alpha_{n}$} [t] at 36 13
    \pinlabel {$\alpha_{1}$} [t] at 146 13
    \pinlabel {$\alpha_{n}$} [t] at 256 13
    \pinlabel {$\alpha_{n}$} [t] at 292 13
    \pinlabel {$\alpha_{1}$} [t] at 402 13
    \pinlabel {$\alpha_{n}$} [t] at 512 13
    \pinlabel {$\alpha_{n}$} [b] at 36 87
    \pinlabel {$\alpha_{1}$} [b] at 146 87
    \pinlabel {$\alpha_{n}$} [b] at 256 87
    \pinlabel {$\alpha_{n}$} [b] at 292 87
    \pinlabel {$\alpha_{1}$} [b] at 402 87
    \pinlabel {$\alpha_{n}$} [b] at 512 87
    \endlabellist
    \includegraphics[width=\textwidth]{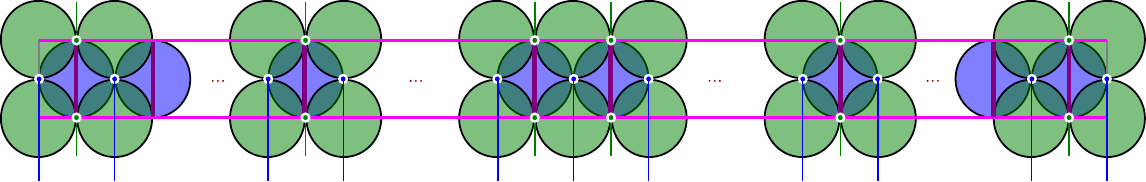}\vspace{.5em}
    \caption{Octahedra visible in a Ford domain; $ \chi $ is at $ \infty $. The vertical translation is the side-pairing that produces all the embedded pink triangles.\label{fig:octahedra}}
  \end{subfigure}
  \caption{Octahedral decomposition of $ B^n$.}
\end{figure}

\begin{thm}\label{thm:octahedra}
  Let $ \mf{b}^n $ be the $n$-stacked Borromean rings, and let $ B^n = \Sph^3 \setminus \mf{b}^n $. Then $B^n$ admits a complete hyperbolic structure that
  decomposes as a union of $4n-2 $ ideal right-angled octahedra.
\end{thm}
\begin{proof}
  The $2$-sphere $S$ in $ \Sph^3 $ containing all the cusps $\mu_j$ (the blue plane of \zcref{fig:octahedra_link}) can be divided
  into a union of $n-1$ four-punctured spheres and $2$ thrice-punctured spheres; a pair of adjacent four-punctured spheres in this
  decomposition is shown in \zcref{fig:octahedra_s04s}. In the same figure, we show the construction of a single octahedron, with
  one vertex at $ \infty $, one each at the cusps $ \mu_{k-1} $, $ \mu_{k} $, and $ \mu_{k+1} $, and four remaining vertices
  coming from meeting each of the two cusps $ \alpha_k $ and $ \alpha_{k+1} $ twice. This octahedron lifts to $ \H^3 $ as shown
  in \zcref{fig:octahedra_single}. These octahedra cover the entire link complement, except for a piece inside the cusp $ \mu_1 $
  and a piece outside the cusp $ \mu_n $. The region inside $ \mu_1 $ is formed from a gluing of four octahedra
  as in \zcref{fig:octahedra_s03s}. The region outside $ \mu_n $ is formed from a single pair of octahedra which glue up around the $\mu_n$ cusp
  as in \zcref{fig:octahedra_s03s_outer}.

  In total, we obtain a set of $4n-2$ octahedra which lift as shown in \zcref{fig:octahedra}. The rectangle is a fundamental domain for the
  central cusp of $ \mf{b}^n $, i.e.\ the cusp $ \chi $. The other circles shown are all of equal radius,
  and the intersection points of these circles are the lifts of cusps as labelled.

  This combinatorial polyhedron can be metrically realised in $ \H^3 $ with all ideal vertices and all angles $ \pi/2 $ or (at adjacency between two octahedra) $ \pi $. It can be checked that
  in the decomposition of $ B^n $ every edge is surrounded by $4$ octahedra; thus all angle sums around edges are $ 2\pi $. In addition the transformations pairing tangent circles can be chosen
  to be parabolic. These observations verify the hypotheses of the Poincar\'e polyhedron theorem and so this ideal triangulation defines a complete hyperbolic structure.
\end{proof}

As a straightforward consequence of \zcref{thm:octahedra} we see that the links $ \mf{b}^n $ are arithmetic~\cite[Lemma~3.2]{pinsky23} with
invariant trace field $ \Q(i) $, and the manifolds $ B^n $ have hyperbolic volume $ (32n-16) \cyrL(\pi/4) $ where $ \cyrL $ is the Lobachevsky function~\cite[\S 7.2]{thurstonN}.
More important to our applications is the following additional corollary.

\begin{cor}
  The cusp shapes of $ B^n $ (i.e. the Euclidean length of the longitude when the meridian is normalised to length $1$) are: $ 4n - 2 $ for $ \chi $;
  $ 2 $ for $ \alpha_1 $ and $ \mu_n $; and $ 1 $ for all other cusps.
\end{cor}
\begin{proof}
  For an ideal regular octahedral decomposition, the cusp shape can be computed as the ratio between the number of polyhedral walls crossed while
  walking around the longitude of each cusp to the number of walls crossed while walking around the meridian. This is carefully explained in Futer and Purcell~\cite[\S 2.2]{futer07}.
  As an example, consider the cusps $ \alpha_k $. When $ k > 1 $, the incident octahedra are as shown in \zcref{fig:octahedra_s04s}. Walking around the intersection point of $ \alpha_k $
  in the horizontal plane $ S $, we meet an orange edge (joining $\alpha_k$ to $\alpha_{k+1} $), a purple edge (joining $\alpha_k$ to $ \alpha_k$), a second orange
  edge (joining $\alpha_k$ to $\alpha_{k-1} $), and a second purple edge (joining $\alpha_k$ to $\alpha_k$), so in walking around the meridian
  of the cusp four octahedra walls are passed. Walking around the longitude one passes through only two octahedra, but there are two additional
  walls (pink) that correspond to a self-gluing of an octahedron, so the total is still $4$ and the cusp shape is $4/4 = 1 $. When $ i = 1 $, the
  incident octahedra are as shown in \zcref{fig:octahedra_s03s}; there are still $4$ octahedra along the longitude, but when walking along the meridian one only passes through
  two walls, giving a cusp shape of $ 4/2 = 2 $.
\end{proof}

\begin{cons}\label{cons:cut_glue}
  The fundamental domain shown in \zcref{fig:octahedra} has a single vertex that descends to $ \chi $, which is placed at $ \infty $, and
  has $4$ vertices which descend to $ \alpha_1 $. We will rearrange the component octahedra so that there is a single vertex which
  descends to $ \alpha_1 $, and we will move it by isometry in $ \H^3 $ so that this single vertex is at $ \infty $.

  Consider the rectangle of \zcref{fig:octahedra} which gives a fundamental domain for the stabiliser of $ \infty $; it is superimposed
  on top of a pair of dual circle packings of the plane. We may translate the rectangle around on these circle packings arbitrarily,
  and the resulting figure will still define a fundamental domain. Translate it so that one of the points descending to $ \alpha_1 $ is
  the centre of reflective symmetry. The resulting domain has $3$ vertices which descend to $ \alpha_1 $,
  which are the points of tangency of the six shaded circles in \zcref{fig:cut_glue_1}. Without loss of generality, the central point
  descending to $ \alpha_1 $ is $ 0 \in \C $ and the grey dotted circle shown in \zcref{fig:cut_glue_1} (defined by being centred at $0$
  and tangent to the four circles meeting there) is the unit circle. We also distinguish two vertical lines (black dotted in the figure),
  namely the vertical lines through the lifts of $ \mu_n $.

  Apply the hyperbolic isometry $ z \mapsto 1/z $; this is a circle inversion in the unit circle followed by reflection in the horizontal line of
  symmetry, and the result is shown in \zcref{fig:cut_glue_2}. This has the effect of moving the vertex corresponding to $ \chi $ to $0$, and moving
  a vertex corresponding to $ \alpha_1 $ to $ \infty $.

  Finally, cut along the hyperbolic planes above the black dotted circles and glue these pieces along the inverts of the green circles, to form a rectangle of aspect ratio $2:1$
  as in \zcref{fig:cut_glue_3}; this is a rearrangement of the octahedra making up the fundamental domain so that the cusp $\alpha_1$ is represented
  by a single vertex. In the resulting domain the cusp $ \chi $ is represented by two vertices.
\end{cons}

\begin{figure}
  \centering
  \begin{subfigure}{\textwidth}\vspace{0pt}
    \centering
    \includegraphics[width=\textwidth]{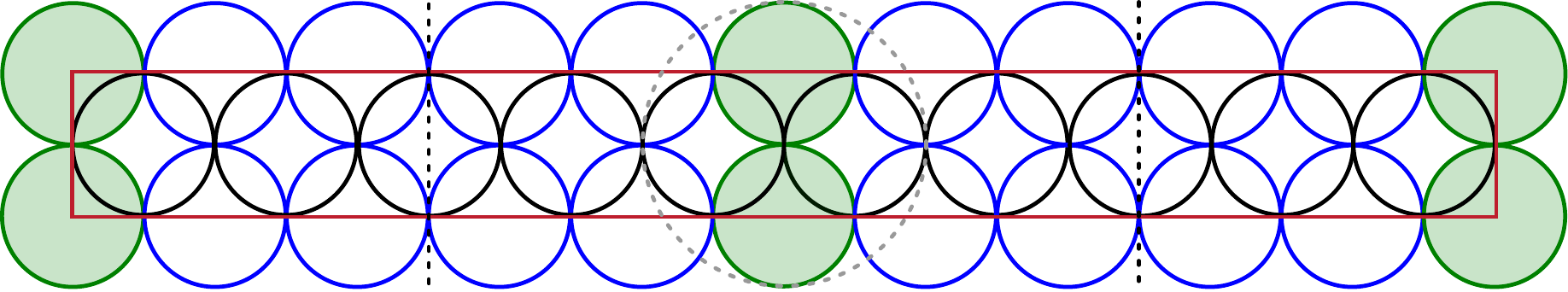}
    \caption{\label{fig:cut_glue_1}}
  \end{subfigure}\\[.5em]
  \begin{subfigure}{.6\textwidth}
    \centering
    \includegraphics[height=.25\textheight]{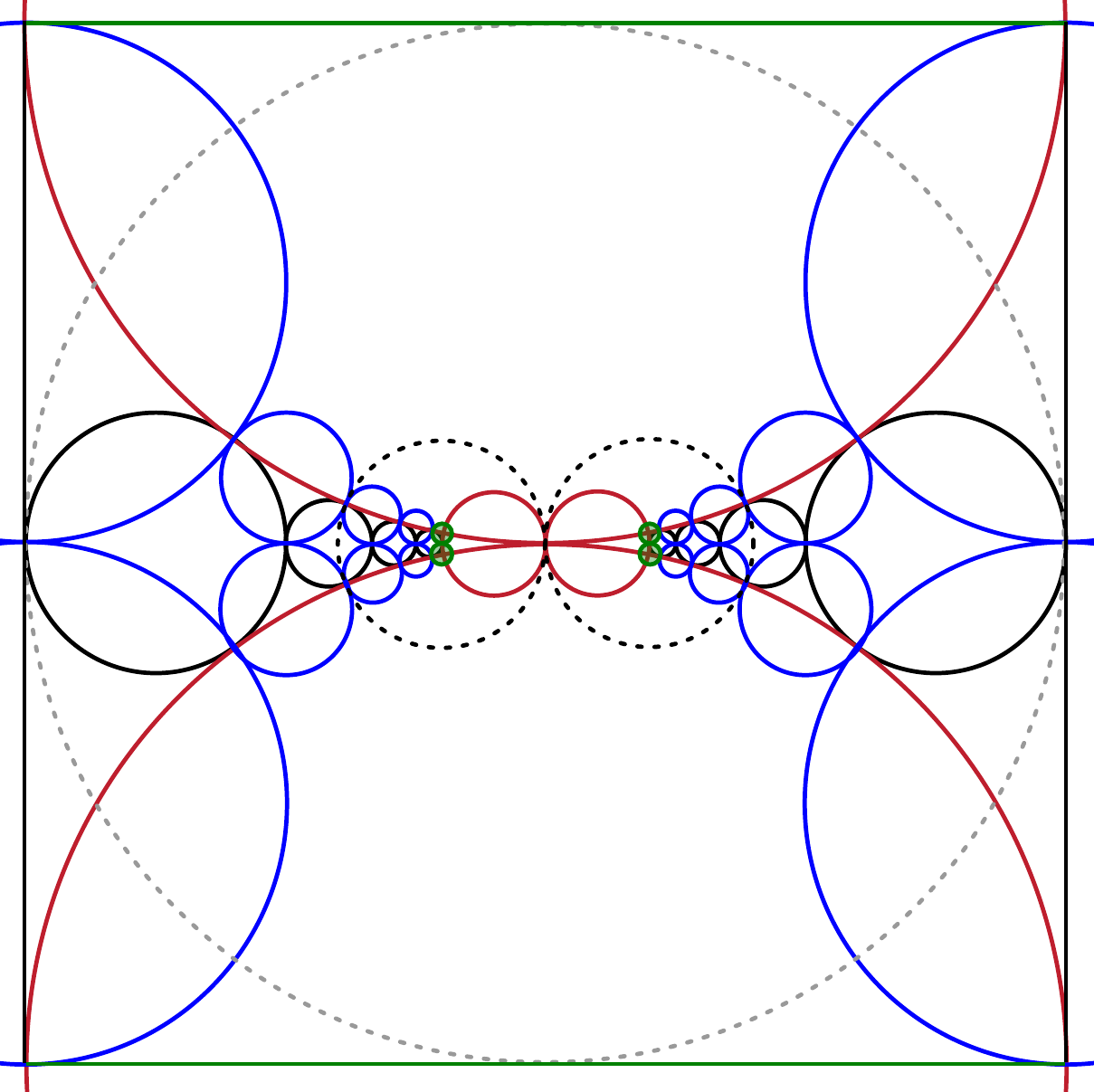}
    \caption{\label{fig:cut_glue_2}}
  \end{subfigure}\hfill
  \begin{subfigure}{.4\textwidth}
    \centering
    \includegraphics[height=.25\textheight]{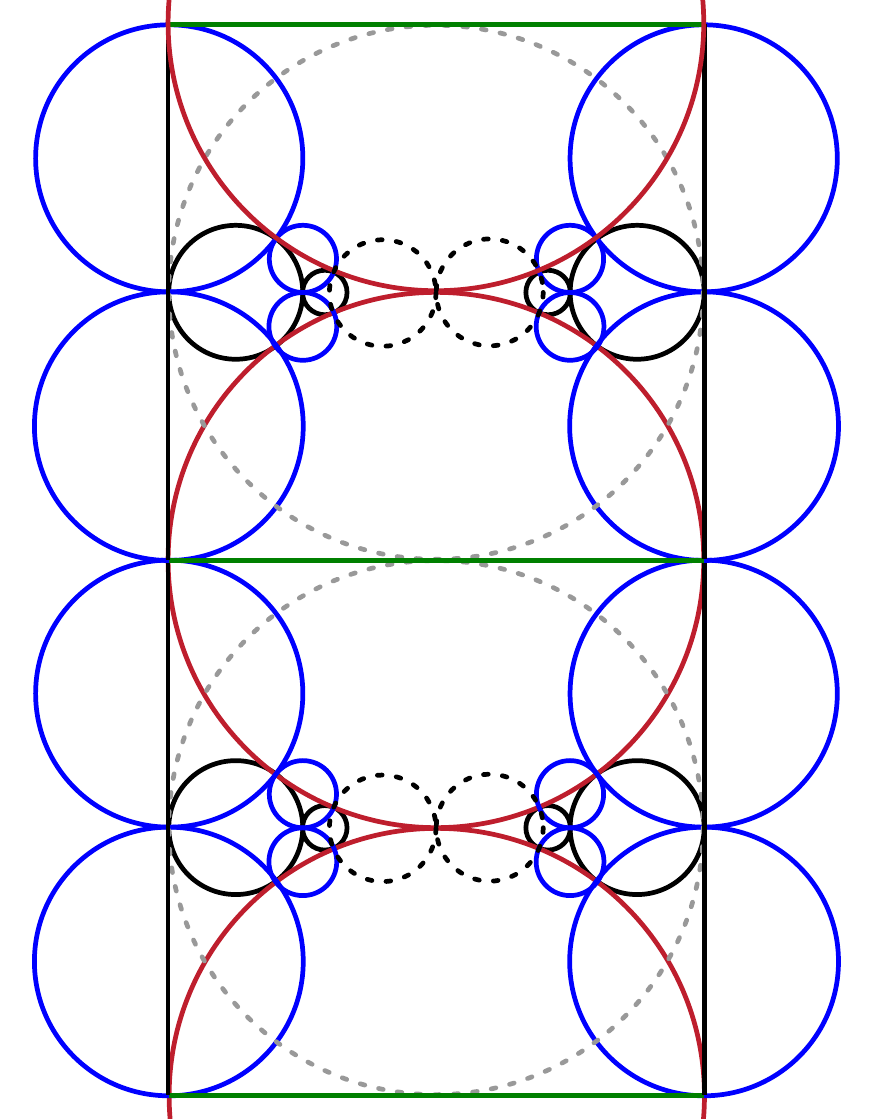}
    \caption{\label{fig:cut_glue_3}}
  \end{subfigure}
  \caption{Cutting and gluing in \zcref{cons:cut_glue}.}
\end{figure}

The domain which results from \zcref{cons:cut_glue} can be compared with  the horoball view of the $ 3$-stacked Borromean
rings displayed in \zcref{fig:borromean_7cpt_cusps}.

\begin{thm}\label{thm:cone_deform_complex}
  There exists a path in the character variety of $ M^n $ parameterising hyperbolic cone manifold holonomy groups in which $ ZY^{-1} $ changes
  from a parabolic, through elliptics of increasing holonomy angle from $ 0 $ to $ 2\pi $, to the identity; at each point the cone
  manifold is obtained by filling the $ ZY^{-1} $ cusp with an ideal cone arc of the appropriate angle, and the path limits onto $ B^n $. This
  deformation preserves all cusp shapes except for that of the cusp $ \alpha_1 $ which, in the limit, has normalised longitude $ 2n-1 $.
\end{thm}
\begin{proof}
  Consider a fundamental domain for $ B^n $ as constructed in \zcref{cons:cut_glue}; the vertical line of symmetry of this figure cuts the two
  vertices that descend to the cusp we wish to remove. Just as in \zcref{prp:cone_deform_simple}, we may cut along this line and pull the two halves
  apart, introducing a new quadrilateral ideal face where each of these vertices was. This corresponds to increasing the translation
  length of the horizontal Euclidean transformation. The geometry around no other edge is affected, so all other edge-cycles remain
  unchanged and the group defined by the side-pairings of the new polyhedron is the holonomy group of a hyperbolic polyhedron with the same cusp
  shapes and combinatorics, the only difference being a new elliptic arc joining the two new thrice-marked spheres at infinity. In particular all
  cusp shapes remain the same since relative sizes of isometric circles of parabolics are preserved (the pairs of circles are simply shifted by
  an appropriate Euclidean translation).
\end{proof}

\begin{rem}\label{rem:hk_deformation}
  As pointed out by a referee, an alternative proof of the existence of the cone deformation in \zcref{thm:cone_deform_complex} can be obtained
  via the following argument. There is an order two rotation $ \Phi \in \Isom^+(M^n) $ with axis along the drilled out tunnel $\tau$. Consider the quotient
  $ M^n/\langle \Phi \rangle $; this is an orbifold with orientation-preserving isotropy groups which has two conformal boundary components, both
  being spheres with two punctures and one angle $ \pi $ cone point. Produce a new orbifold $ \overline{M^n} $ by doubling $M^n/\langle \Phi \rangle$
  across these boundary components. This orbifold admits a finite volume hyperbolic metric by gluing two copies of the hyperbolic convex core of $M^n/\langle \Phi \rangle$
  together. The desired cone deformation of $ M^n $ which modifies the angle around $\tau$ from $0$ to $2\pi$ exists if and only if the image of $ \tau $ in $ \overline{M^n} $
  can be cone-deformed from $0$ to $\pi$. This can be done by Wei\ss' global rigidity theorem~\cite{weiss07}, which states that any cone deformation of a finite volume
  hyperbolic cone manifold can be performed in the realm $ (0,\pi) $, and so the argument is complete.

  This argument bypasses the explicit construction of fundamental domains of the intermediate cone manifolds, but unfortunately it cannot prove
  that the shapes of cusps are preserved. Purcell~\cite{purcell08} has studied cusp shapes of hyperbolic manifolds under Hodgson--Kerckhoff type
  deformations like these: one can obtain bounds on the change on cusp shape based on the injectivity radius of the cone arc (which is in
  principle possible to compute from the fundamental domain of \zcref{thm:octahedra}), but these do not allow us to conclude that the shapes remain
  constant. Knowing the cusp shapes is important for Dehn filling applications, but if these shapes are not needed then the more abstract argument
  is slightly cleaner.

  In addition, this argument unfortunately cannot bypass the most technical part of this section, namely the proof that the manifolds $M^n $ or $ B^n $ are
  actually hyperbolic and geometrically finite (\zcref{thm:octahedra}), which is needed to conclude that $ \overline{M^n} $ is finite volume. Its main
  advantage is that it bypasses the need to rearrange all the octahedra as in \zcref{cons:cut_glue} to get them into an adventageous shape for deforming
  the fundamental polyhedron by hand.
\end{rem}

\begin{rem}[The limits $ B^\infty $ and $ M^\infty $]
  Two examples of non-free infinitely generated Kleinian groups arise as limiting cases. Indeed, the Poincar\'e polyhedron theorem holds for polyhedra in $\H^3$ with countably many
  faces, so long as ideal vertices are locally finite; see \zcref{sec:poincare}. As $ n \to \infty $, the polyhedra of \zcref{cons:cut_glue} have a well-defined limit, and this limit
  satisfies the conditions of the polyhedron theorem: the face-pairings generate discrete groups and in each case the polyhedron glues
  to give a complete hyperbolic metric of infinite volume on the quotients $ M^\infty $ and $ M^\infty $.

  The group $ \Hol(B^\infty) $ includes a rank $1$ parabolic subgroup which is not doubly cusped. The group $ \Hol(M^\infty) $ includes a pair of thrice-punctured sphere groups
  which share two conjugacy classes of maximal rank $1$ parabolic subgroups; these groups are doubly cusped. The third puncture of each thrice-punctured sphere is represented
  by a rank $1$ parabolic group that is not doubly cusped.
\end{rem}

\subsection{Dehn filling}\label{sec:filling}
The Gromov--Thurston $2\pi$-theorem, see Bleiler and Hodgson~\cite[Theorem~9]{bleiler96}, states that if $M$ is a complete hyperbolic $3$-manifold
with finite volume then a sufficient condition for a Dehn filling of $M$ to be negatively curved is that the Euclidean length of the geodesic representing
each slope is at least $2\pi$. With this machinery, we may deduce the following result.

\begin{thm}\label{cor:unknotting_tunnel}
  Let $ \mf{k} $ be a $2$-bridge knot so that the slope of $ \mf{k} $ has continued fraction decomposition $ [2m_n, 2a_n, \ldots, 2m_1, 2a_1] $.
  If $ m_n \geq 4 $ and $  a_n, m_{n-1}, \ldots, m_1, a_1 \geq 7 $, then there is a continuous family of negatively curved cone manifolds $ M_\theta $, $ \theta \in (0,2\pi] $,
  supported on the link complement $ \Sph^3 \setminus \mf{k} $ with singular locus consisting exactly of an ideal cone arc along the upper unknotting tunnel of $ \mf{k} $.
\end{thm}
\begin{proof}
  Suppose that $ \mf{k} $ has slope $ p/q $. Consider a genus $2$ handlebody $ \mc{H} $; let $ \sigma_1 $ and $ \sigma_2 $ be projections to the genus $2$ surface $ \partial \mc{H} $ of the
  cores of the two handles. Then $ S = \partial \mc{H} \setminus (\sigma_1 \union \sigma_2) $ is a four-holed sphere. Let $ \tau $ be the simple closed curve on $S$
  with homology class $ p \gamma_{1/0} + q \gamma_{0/1} $, where $ \gamma_{1/0} $ is the homology of the curve on $S$ bounding a compression disc in $ \mc{H} $ and where $ \gamma_{0/1} $
  is a curve which intersects $ \gamma_{1/0} $ exactly twice; c.f.\ Farb and Margalit~\cite[\S 2.2.5]{farb}. Then $ \Sph^3 \setminus \mf{k} $ is the manifold obtained
  by gluing a 2-handle onto $ \mc{H} $ along $ \tau $. In addition, the continued fraction decomposition $ [2m_n, 2a_n, \ldots, 2m_1, 2a_1] $ encodes the sequence of Dehn twists
  of the four-punctured sphere which produce $ \tau $ from $ \gamma_{0/1} $. There is a hyperbolic $3$-manifold $ N $ on the boundary of genus $2$ Schottky space which
  is homeomorphic to $ \mc{H} $ and which has rank $1$ cusps represented by the three loops $ \sigma_1 $, $ \tau$, and $ \sigma_2 $. Consider now the lantern manifold $ M^n $. If we Dehn
  fill $ M^n $ by $ 1/a_i $ filling the cusps $ \alpha_i $ and $ 1/m_j $ filling the cusps $ \mu_j $ then we obtain $N$: indeed, the various Dehn fillings of slope $ 1/a_i $
  (resp. $ 1/m_j $) just perform $a_i$ (resp. $m_j$) Dehn twists along $ \gamma_{1/0} $ (resp. $ \gamma_{0/1} $). The point is that $N$ is obtained by drilling the upper unknotting tunnel
  from $ \Sph^3 \setminus \mf{k} $ and the loop $ \tau $ is the meridian of the unknotting tunnel. By the $2\pi$-theorem, and using the cusp shapes computed in \zcref{thm:cone_deform_complex},
  the filling is hyperbolic whenever the conditions in the theorem statement are satisfied. Explicitly, $ \sqrt{1 + (2n)^2} > 2\pi $ if and only if $ n \geq 4 $ and $ \sqrt{1 + n^2} > 2\pi $
  if and only if $ n \geq 7 $.

  Next, we note that the proof of the $2\pi$ theorem given by Blieler and Hodgson~\cite[Theorem~9]{bleiler96} works without change for cone manifolds: it relies
  only on a lemma of Gromov (Lemma~10 \textit{op.\ cit.}) which fills in the hyperbolic metric on the interior of a solid torus to extend a hyperbolic metric on a neighbourhood
  of the boundary, and this is independent of whether there are singular arcs away from the horoball being filled. It is also clear from the proof of this
  lemma that the metric in the filled horoball is continuous with respect to the metric near the boundary, so continuous variation of the metric in the manifold being
  filled will induce a continuous metric on the filled manifolds.

  Now we can put these observations together: if $ M^n_\theta $ is the cone manifold obtained by cone-deforming $ M^n $ using \zcref{thm:cone_deform_complex} so that
  the cone arc has angle $ \theta $, then we can use the $2\pi$ theorem to Dehn fill it and obtain a cone manifold $ N_\theta $ supported on $ \Sph^3 \setminus \mf{k} $
  with a cone arc of angle $ \theta $ along the upper unknotting tunnel.
\end{proof}

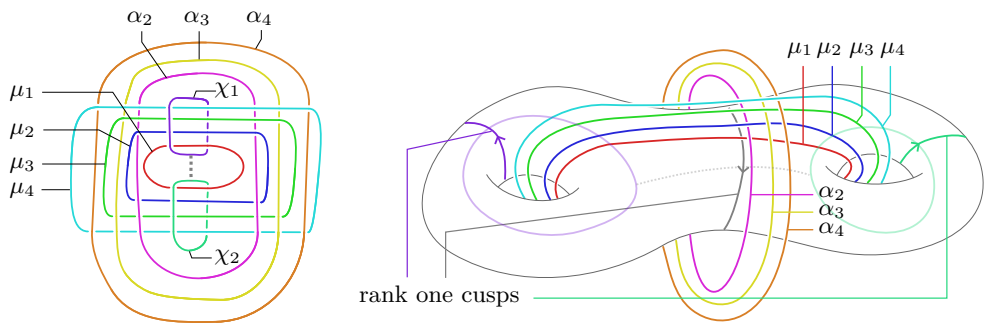
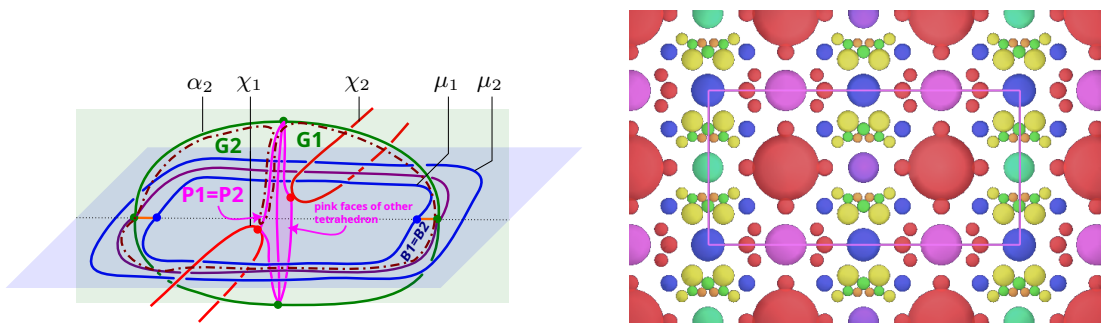
\begin{figure}
  \centerline{\begin{subfigure}[t]{0.33\textwidth}\centering
    \labellist
    \small\hair 2pt
    \pinlabel {$\mu_1$} [r] at 44 280
    \pinlabel {$\mu_2$} [r] at 44 233
    \pinlabel {$\mu_3$} [r] at 44 192
    \pinlabel {$\mu_4$} [r] at 44 161
    \pinlabel {$\chi_1$} [l] at 247 283
    \pinlabel {$\chi_2$} [l] at 247 80
    \pinlabel {$\alpha_2$} [b] at 164 356
    \pinlabel {$\alpha_3$} [b] at 233 356
    \pinlabel {$\alpha_4$} [b] at 308 356
    \endlabellist
    \includegraphics[width=\textwidth]{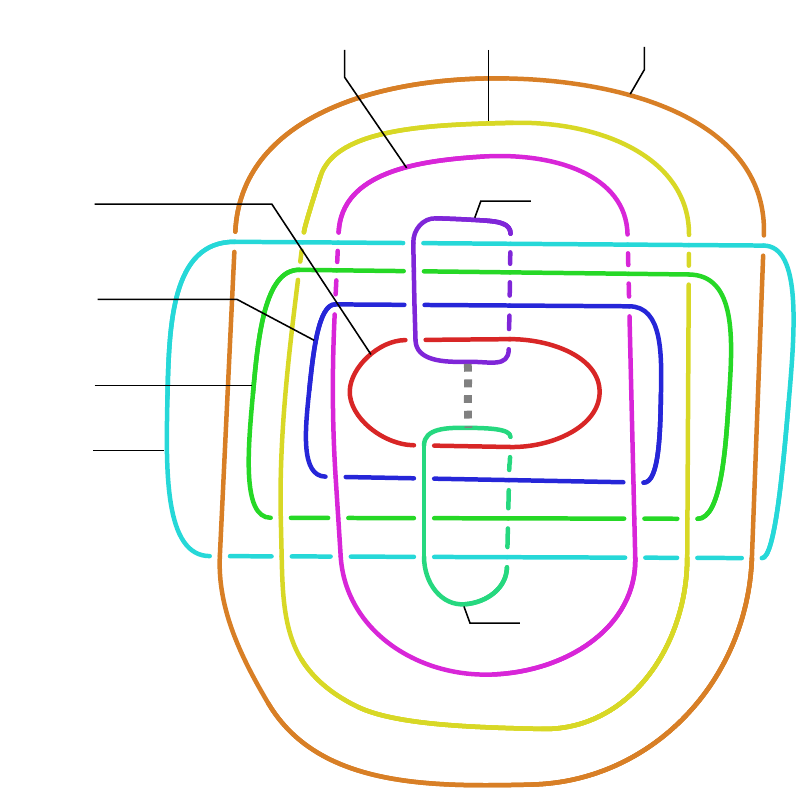}
    \caption{A link analogous to $ \mf{b}^4 $ but with bifurcated central cusp. Also shown (dotted) is the embedded arc which is drilled out by cone deformation.\label{fig:non_borromean_link}}
  \end{subfigure}\hspace{1cm}
  \begin{subfigure}[t]{0.6\textwidth}\centering
    \labellist
    \small\hair 2pt
    \pinlabel {rank one cusps} at 20 17
    \pinlabel {$\alpha_2$} [l] at 249 80
    \pinlabel {$\alpha_3$} [l] at 249 69
    \pinlabel {$\alpha_4$} [l] at 249 58
    \pinlabel {$\mu_1$} [b] at 241 161
    \pinlabel {$\mu_2$} [b] at 259 161
    \pinlabel {$\mu_3$} [b] at 279 161
    \pinlabel {$\mu_4$} [b] at 298 161
    \endlabellist
    \includegraphics[width=\textwidth]{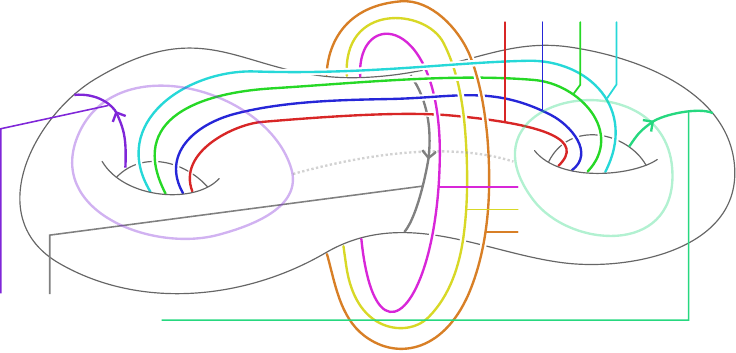}
    \caption{A manifold analogous to $M^4$ but with a different pattern of rank $1$ cusps. The central cusps $ \chi_1 $ and $ \chi_2 $ have become closed loops on the genus $2$ surface, each dual to a rank $1$ cusp. The cone arc is dual to the third rank $1$ cusp.\label{fig:non_borromean_lantern}}
  \end{subfigure}}\vspace{.5cm}
  \centerline{\begin{subfigure}[t]{.6\textwidth}\centering
    \labellist
    \small\hair 2pt
    \pinlabel {$\alpha_2$} [b] at 186 217
    \pinlabel {$\chi_1$} [b] at 233 217
    \pinlabel {$\chi_2$} [b] at 336 217
    \pinlabel {$\mu_1$} [b] at 421 217
    \pinlabel {$\mu_2$} [b] at 462 217
    \endlabellist
    \includegraphics[width=\textwidth]{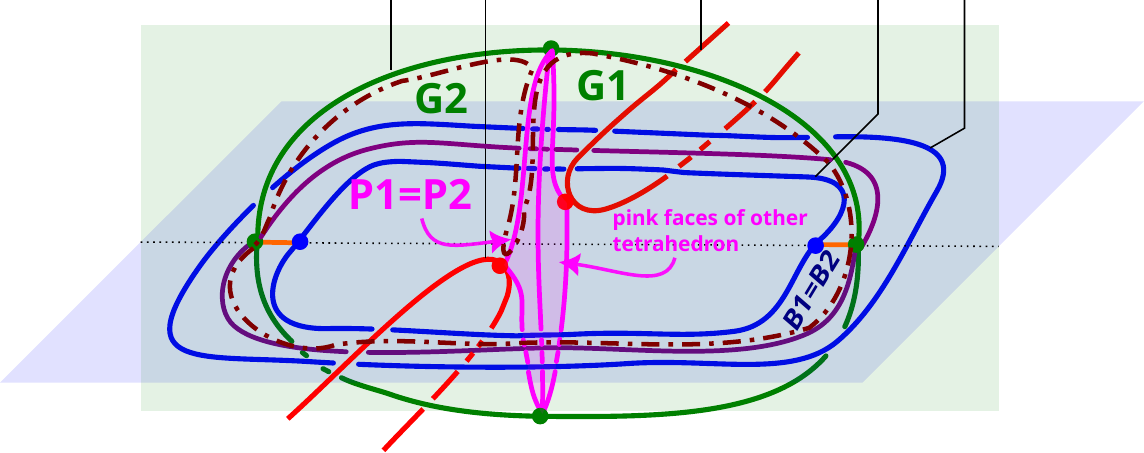}
    \caption{The two octahedra that lie inside $ \mu_1 $. Edge colours match \zcref{fig:octahedra_single}, and just as in \zcref{fig:octahedra_s03s_outer} we have also outlined
    one of the unshaded faces.\label{fig:octahedra_s03s_nonbor}}
  \end{subfigure}\hspace{.5cm}
  \begin{subfigure}[t]{0.49\textwidth}
    \includegraphics[width=\textwidth]{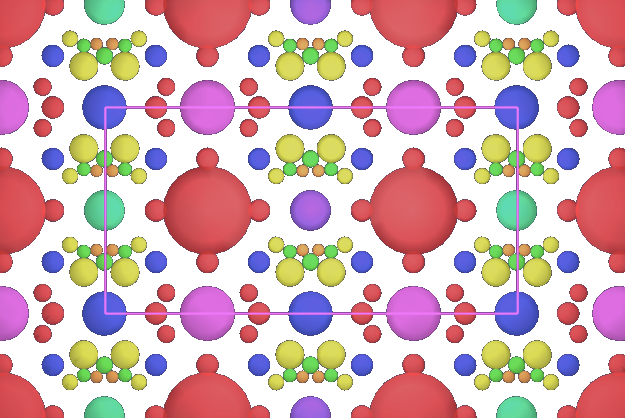}
    \caption{Cusp horoballs with $\mu_1$ at $ \infty $.\label{fig:non_borromean_2to1_cusp}}
  \end{subfigure}}
  \caption{The manifolds analogous to $ B^n $ and $ M^n $ that can be used to study $2$-component links.\label{fig:nonborromean}}
\end{figure}

Replacing the stacked Borromean rings with a slightly different link, we can obtain a version of \zcref{cor:unknotting_tunnel} for links of two components. We will only provide
a sketch of the proof as it requires no new ideas.
\begin{prp}\label{prp:unknotting_tunnel_link}
  Let $ \mf{k} $ be a $2$-bridge link with $2$ components whose slope has continued fraction decomposition $ [2m_n, 2a_n, \ldots, 2m_1] $. If $ m_n \geq 4 $ and $ a_n, m_{n-1}, \ldots, m_1 \geq 7 $,
  then there is a continuous family of negatively curved cone manifolds $M_\theta$, parameterised by $\theta\in (0,2\pi]$, supported on the link complement $ \Sph^3 \setminus \mf{k} $ and with singular locus consisting
  exactly of an ideal cone arc of angle $\theta$ along the upper unknotting tunnel of $ \mf{k} $.
\end{prp}
\begin{proof}
  Take the central component of $ \mf{b}^n $ and split it into two loops, $ \chi_1$ and $ \chi_2 $, as in \zcref{fig:non_borromean_link}; the component $ \alpha_1$ is
  now be homotopic to a point, and so we throw it away. The complement of this link in $ \Sph^3 $ has an octahedral decomposition which can be found, just as
  in \zcref{thm:octahedra}, by placing one of the two central cusps at $ \infty $; the cusp shape of each central cusp is now $(2n-2) $ rather than $(4n-2) $, and the
  arrangement of the octahedra is exactly the same as in the stacked Borromean rings except for the inner four octahedra which are replaced by a pair of octahedra
  glued to `wrap around' just like the outer pair of octahedra---we show these two octahedra in \zcref{fig:octahedra_s03s_nonbor}. If the cusp $ \mu_1 $ to is moved to $ \infty $ then
  we end up with an aspect ratio $2:1$ rectangle as in \zcref{fig:non_borromean_2to1_cusp}. Just as in \zcref{thm:cone_deform_complex} this fundamental domain can be modified
  continuously to `pull apart' the two parabolic fixed points corresponding to the two central cusps of the link; this cone-deforms the link in \zcref{fig:non_borromean_link}
  to the manifold in \zcref{{fig:non_borromean_lantern}} with $2n-1$ rank two cusps and conformal boundary consisting of two thrice-punctured spheres glued along rank $1$ cusps
  along the indicated curves. Running through the proof of \zcref{cor:unknotting_tunnel} with this new setup gives the result.
\end{proof}

\begin{rem}\label{rem:sakuma_conjecture}
  Sakuma and others have conjectured (see Akiyoshi, Wada, Sakuma, and Yamashita~\cite[preface]{akiyoshi}) that there exists a family
  of complete \emph{hyperbolic} metrics which realise the drilling of the upper unknotting tunnel of a hyperbolic $2$-bridge link, and
  Lee and Sakuma~\cite[\S 9]{lee13} sketched a possible approach via a cone-manifold version of the Casson--Rivin theory given
  for two-bridge links by Futer~\cite{gueritaud06}, but the details have not appeared. \zcref[S]{cor:unknotting_tunnel} and \zcref{prp:unknotting_tunnel_link} give
  evidence for this conjecture: they show that for sufficiently complicated links there is such a deformation through
  pinched negatively curved metrics which have constant curvature $-1$ in the region bounded away from the cone singularities. Actually,
  the argument given by the anonymous referee in \zcref{rem:hk_deformation} can be used to prove the conjecture fully.

  Let $ \mf{b}$ be any $2$-bridge link (hyperbolic or not), and let $M_{\mf{b}}$ be the corresponding maximal cusp on the boundary of the Riley slice. Let $h$ be
  the hyperelliptic involution in $\Aut(M_{\mf{b}})$ (see \S 2 of \cite{akiyoshi21} for the action of $h$ on $ \mf{b}$). Then $M_{\mf{b}}/\langle h\rangle$ is an orbifold
  with two thrice-marked spheres ($2$ punctures and one $\pi$-cone point each) on the boundary. Since $M_{\mf{b}}$ is geometrically finite, the orbifold $N_{\mf{b}}$
  obtained by doubling $M_{\mf{b}}/\langle h \rangle $ across its boundary is finite volume with only cusps and loop singularities. In addition, since the axis of $h$
  contained the drilled unknotting tunnel, a deformation of the tunnel through angles $(0,\pi)$ in $N_{\mf{b}}$ corresponds to a deformation through angles $(0,2\pi)$
  in $M_{\mf{b}}$. But by the global rigidity theorem of Weiss~\cite{weiss07}, it is possible to deform any closed geodesic in a finite volume orbifold through cone
  angles in the interval $(0,\pi)$, which completes the proof. If $ \mf{b} $ is hyperbolic in the limit as the angle goes to $2\pi$ then the limiting manifold is the complete
  hyperbolic metric on the manifold $ \Sph^3 \setminus \mf{b} $.

  We note that it is a result of Futer~\cite{futer07b} that such involutions which fix the unknotting tunnel occur only for $2$-bridge links and so this
  argument does not generalise to the full conjecture of Sakuma, which is that cone deformations can be used to drill out the unknotting tunnel
  of \emph{any} tunnel number $1$ link~\cite[x]{akiyoshi}. For further context see Sakuma~\cite[\S 4]{sakuma98}.

  Our by-hand argument is still of interest as it generalises to Dehn fillings of all manifolds containing expansion joints, even those without any order two rotational symmetries. In particular, in the next section
  we show that all fully augmented links admit such cone deformations, and since every link in $ \Sph^3 $ is a Dehn filling of a fully augmented link we obtain cone deformations
  for all links through negatively curved metrics via the methods of \zcref{cor:unknotting_tunnel} and \zcref{prp:unknotting_tunnel_link}; these cannot be obtained
  using the order two rotation trick (c.f.\ \zcref{rem:hk_fal} below).
\end{rem}

\section{Concave lenses and other expansion joints}\label{sec:lens}
We isolate the substructure of the manifold which makes the arguments in the previous sections work.

\begin{defn}
  A \df{concave lens} is a polyhedral complex made up of
  $n$ distinct octahedra $ C_1,\ldots,C_n$ that all share a vertex $v$ and so that $ C_{i-1} \inter C_i $ and $ C_i \inter C_{i+1} $ (subscripts taken modulo $n$)
  are $2$-faces that intersect only at $v$. Roughly speaking, the $1$-skeleton of the complex is a cycle that does `not turn any corners'. The \df{outer vertices}
  of the lens are the $n$ vertices that do not lie in faces incident with $ v $. The \df{inner faces} are the $2n$ boundary faces that are incident with $v$.
  The \df{rim faces} are the $2n$ faces that meet an outer vertex and do not meet an inner face.
\end{defn}

In \zcref{fig:lens} we show a concave lens with $ n = 8 $. We will now give some conditions on a concave lens embedded in a manifold that will allow it to act
as an expansion joint in the sense we described in the introduction: the longitude of the central cusp can be elongated, producing gaps
in the polyhedral decomposition of the manifold around the opposite vertices, while preserving the conformal structure of ends not adjacent to the lens.

\begin{figure}
  \centering
  \includegraphics[width=.5\textwidth]{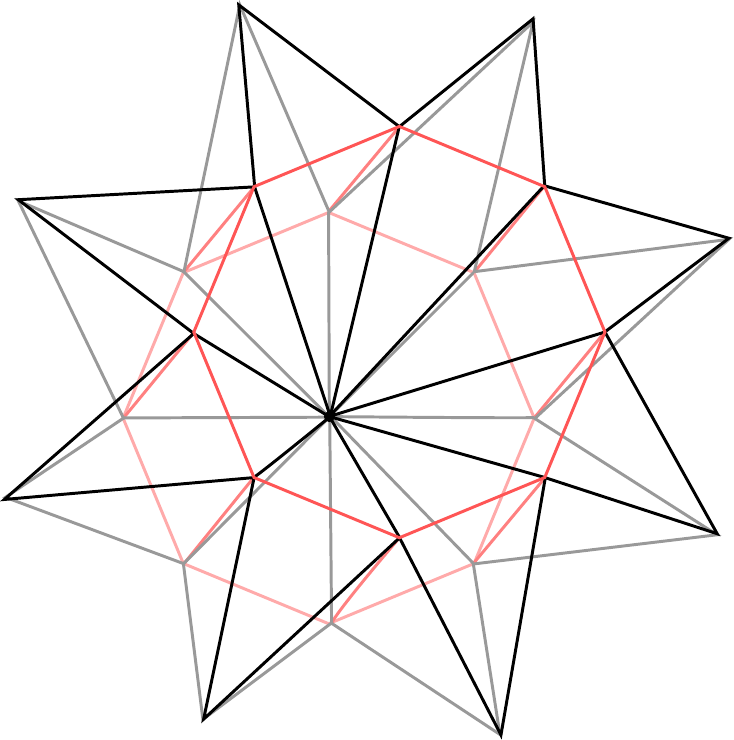}
  \caption{Concave lens of eight octahedra.\label{fig:lens}}
\end{figure}

\begin{thm}\label{thm:concave_lens}
  Let $M$ be a hyperbolic $3$-manifold with convex core $ K(M) $ decomposed as a polyhedral complex $ \mc{C} $. Let $ v $ be an ideal vertex of $ \mc{C} $
  that models a rank $2$ cusp of $M$, and let $ \hat{\mc{C}} $ be a connected lift of $ \mc{C} $ to the universal cover of $ K(M) $ so that $v$ lifts to a
  single vertex. Let $ \sim $ be the equivalence relation on $ \hat{\mc{C}} $ induced by the quotient projection.
  Suppose that $ \hat{\mc{C}} $ contains an embedded concave lens $\mc{L}$ made up of ideal octahedra, with central vertex $v$, satisfying the following properties:
  \begin{enumerate}
    \item The outer vertices of $ \mc{L} $ are lifts of rank $2$ cusps of $M$.
    \item The $\sim$-class of each outer vertex of $ \mc{L} $ consists only of outer vertices of $ \mc{L} $.
    \item The $\sim$-class of each rim face of $ \mc{L} $ consists only of rim faces of $ \mc{L} $.
    \item $ \mc{L} $ is symmetric under reflection in a geodesic plane that contains $ v $ and all the outer vertices.
  \end{enumerate}

  Then there exists a smooth path $M_\theta$ of cone manifolds, parameterised by $ \theta \in (0,2\pi] $ consisting of hyperbolic cone manifolds that are supported on the
  topological manifold $ M_{2\pi} = M $ and where the cone arcs in $ M_\theta $ join the outer vertices of $ \mc{L} $ along the rim faces. If $M$ has any conformal ends,
  then their conformal structure is preserved by the deformation, and the shapes of rank $2$ cusps that do not meet $ \mc{L} $ are preserved.
\end{thm}

\begin{rem}\label{rem:lens_pictures}
  We have arranged all images of cusp horoballs in this paper so that the outer cusps from the relevant concave lens appear along the horizontal line of symmetry. This means
  that the central plane of symmetry of the cusp is the hyperbolic plane through $ \infty $ supported on this horizontal line, and the cone arcs appear along this line (c.f.\ \zcref{fig:borromean_7cpt_cusps}).
\end{rem}

To prove \zcref{thm:concave_lens}, we will need the following standard lemma.
\begin{lem}\label{lem:lattice}
  Let $ \Lambda $ be a rank $2$ lattice. If $ f \in \Lambda $ is primitive, i.e.\ if whenever there exist $ n \in \N $ and $ g \in \Lambda $ with $ n \cdot g = f $
  then $ n = 1 $ and $ g = f $, then $ \{f\} $ may be extended to a basis of $ \Lambda $. \qed
\end{lem}

\begin{proof}[Proof of \zcref{thm:concave_lens}]
  Choose a fundamental domain for $M $ consisting of a lift $ \hat{\mc{C}} $ of $ \mc{C} $ to $ \H^3 $. By an appropriate isometry we may assume that $ v $ lifts to $ \infty $. Let $ f $ be the element
  of $ \Hol(M) $ corresponding to the closed around $ v $ in the concave lens octahedra; without loss of generality, $ f $ is a vertical translation, along the $y$-axis. Since the lens is embedded,
  $ f $ is a primitive element of the rank $2$ lattice $ \Stab_{\Hol(M)} (\infty) $ and hence by \zcref{lem:lattice} we can choose a generating set for this lattice containing $ f $.

  \begin{figure}
    \labellist
    \small\hair 2pt
    \pinlabel {lift of lens} [r] at 95 195
    \endlabellist
    \centering
    \includegraphics[width=.5\textwidth]{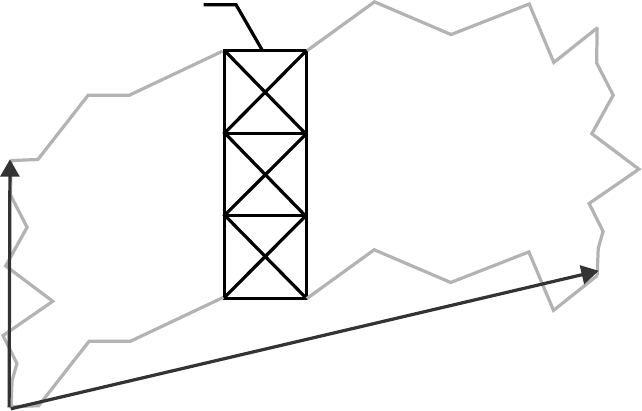}
    \caption{Lift to $ \H^3 $ of a manifold containing a concave lens made up of ideal regular octahedra.\label{fig:lens_lift}}
  \end{figure}

  The fundamental domain for $ M $ coarsely looks like \zcref{fig:lens_lift}; since the lens is embedded, there are no vertices of the complex $ \hat{\mc{C}} $ lying above the rectangular projection of
  the lens octahedra to the Riemann sphere other than $ \infty $ (but it is possible for other vertices to lie below, as seen for instance in \zcref{fig:cut_glue_3} above). Just as
  in \zcref{thm:cone_deform_complex} we may now horizontally pull apart the two halves of this rectangular projection, without changing the combinatorial or geometric structure of any other part of $ \hat{C} $ (in
  particular, we preserve all angle sums away from the lift of the lens). All cusp shapes remain constant except for the shape of the cusp at $ \infty $; the translation vector of $ f $
  remains constant, while the translation vector of the other basis element of $ \Stab_{\Hol(M)} (\infty) $ is skewed (the component parallel to the translation vector of $ f $ remains
  constant and the orthogonal component increases in length by exactly the width of the rectangular projection of the lens octahedra).
\end{proof}

\subsection{Fully augmented links}
We present a family of examples realising \zcref{thm:concave_lens}, arising from fully augmented links.
The components of a fully augmented link are separated into \df{augmentation cusps} (also known as \df{crossing circles}) which each bound twice-punctured discs
normal to some plane of reflective symmetry, and \df{planar cusps} (also known as \df{knot strands}) that lie in the reflection plane; for more
information see Purcell~\cite{purcell09}. Every planar cusp in a fully augmented link bounds two multiply-punctured discs in the reflection plane.

\begin{thm}\label{prp:fal}
  Let $ \mf{l} $ be a fully augmented link with a planar cusp $ \omega $ such that:
  \begin{enumerate}[label={(\roman*)}]
    \item There are no self-augmentations of $ \omega $ (i.e.\ no augmentation cusp bounds a disc that meets $\omega$ twice), and\label{cond:fal1}
    \item One of the two discs, $D$, bounded by $ \omega $ in the reflection plane of $ \mf{l}$ must be punctured only by augmentation cusps.\label{cond:fal2}
  \end{enumerate}
  Then $\omega$ forms the centre of a concave lens satisfying the conditions of \zcref{thm:concave_lens}, and increasing the longitudinal length of $\omega$
  forms cone arcs in the geodesic surface $D$ joining the augmentation cusp punctures together.
\end{thm}

\begin{figure}\centering
  \begin{subfigure}[c]{0.4\textwidth}\centering
    \centering
    \labellist
    \small\hair 2pt
    \pinlabel {$\ast$} [l] at 88 224
    \pinlabel {$\dagger$} [l] at 132 224
    \pinlabel {$\ddagger$} [r] at 182 224
    \endlabellist
    \includegraphics[height=.2\textheight]{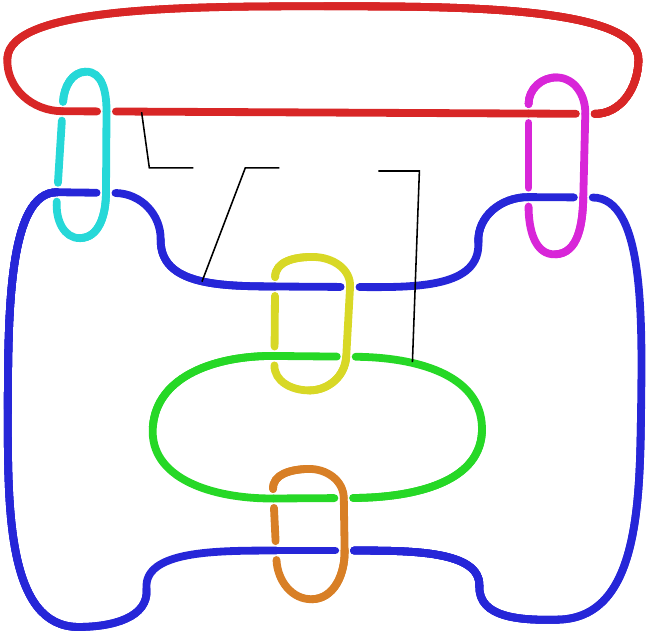}
    \caption{Link in $ \Sph^3$.\label{fig:robot_link}}
  \end{subfigure}\hfill%
  \begin{subfigure}[c]{0.55\textwidth}\centering
    \includegraphics[height=.2\textheight]{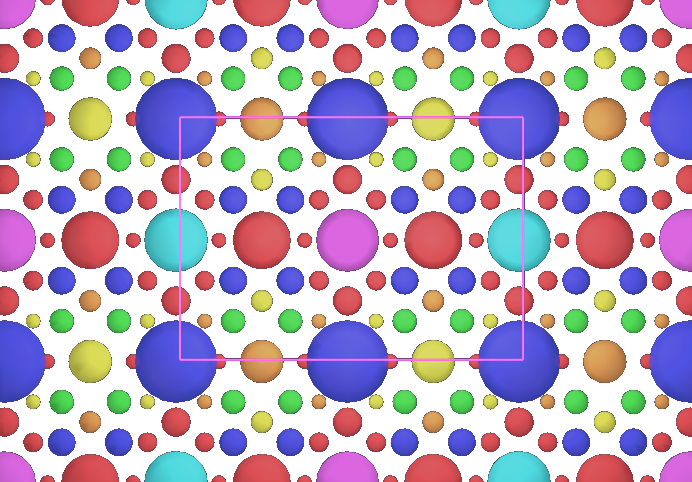}
    \caption{Cusp horoballs with $\ast$ at $ \infty $.\label{fig:robot_cusps_red}}
  \end{subfigure}\\[.5em]%
  \begin{subfigure}[c]{0.45\textwidth}\centering
    \includegraphics[height=.15\textheight]{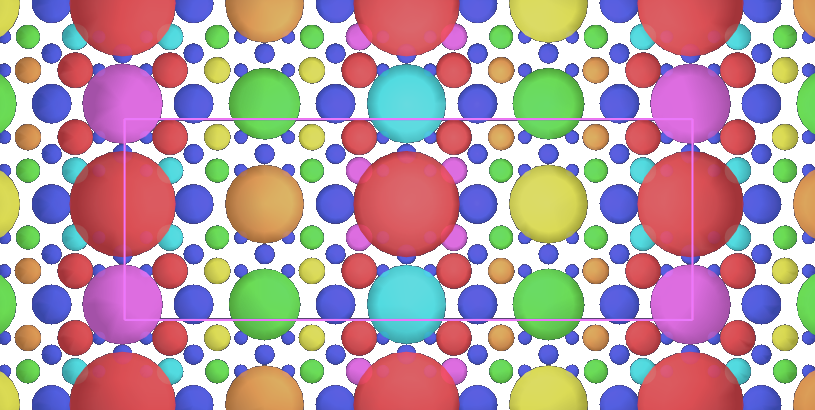}
    \caption{Cusp horoballs with $\dagger$ at $ \infty $.\label{fig:robot_cusps_blue}}
  \end{subfigure}\hfill%
  \begin{subfigure}[c]{0.45\textwidth}\centering
    \includegraphics[height=.15\textheight]{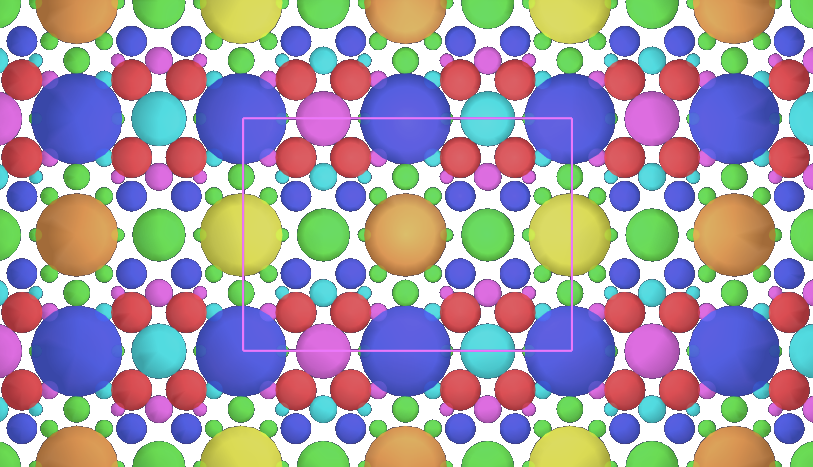}
    \caption{Cusp horoballs with $\ddagger$ at $ \infty $.\label{fig:robot_cusps_green}}
  \end{subfigure}
  \caption{An octahedral fully augmented link.\label{fig:robot}}
\end{figure}

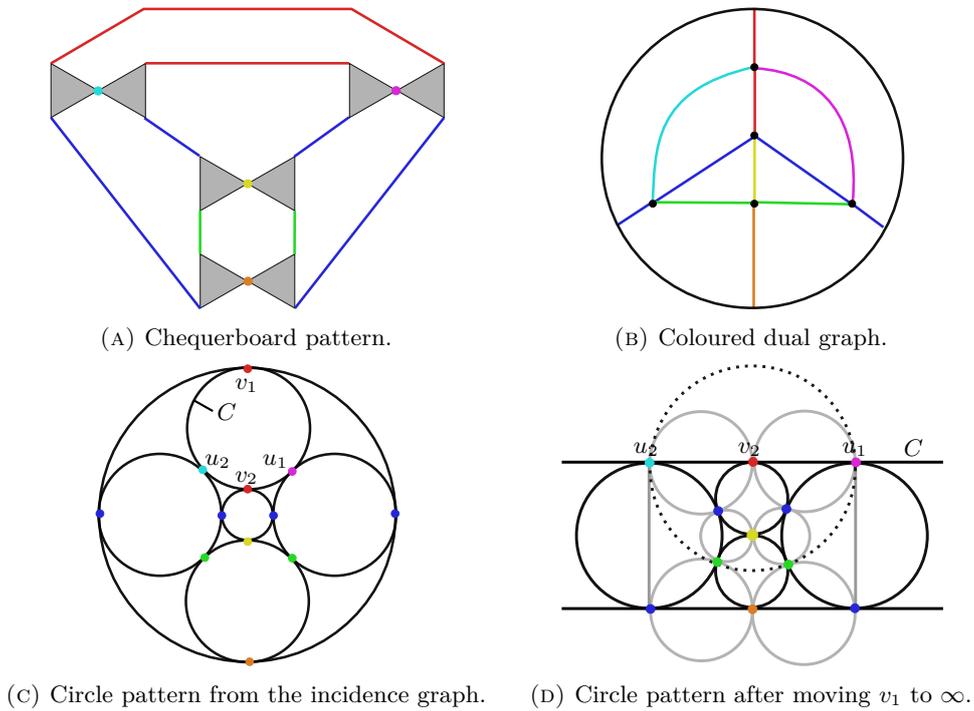
\begin{figure}
  \centerline{\begin{subfigure}[c]{0.5\textwidth}\centering
    \includegraphics[height=.18\textheight]{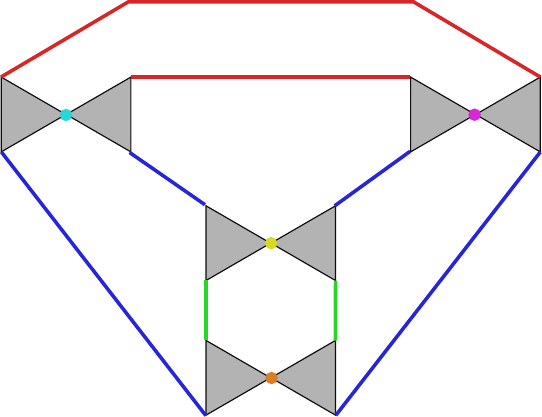}
    \caption{Chequerboard pattern.\label{fig:chequerboard1}}
  \end{subfigure}\hspace{.2cm}
  \begin{subfigure}[c]{0.5\textwidth}\centering
    \includegraphics[height=.18\textheight]{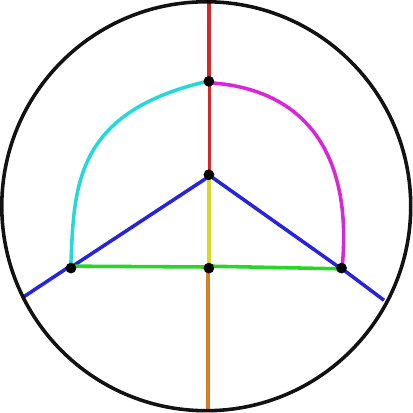}
    \caption{Coloured dual graph.\label{fig:chequerboard2}}
  \end{subfigure}}\vspace{.1cm}
  \centerline{\begin{subfigure}[c]{0.5\textwidth}\centering
    \labellist
    \small\hair 2pt
    \pinlabel {$v_1$} [t] at 100 195
    \pinlabel {$u_1$} [br] at 130 130
    \pinlabel {$v_2$} [b] at 100 118
    \pinlabel {$u_2$} [bl] at 70 130
    \pinlabel {$C$} [l] at 77 170
    \endlabellist
    \includegraphics[height=.18\textheight]{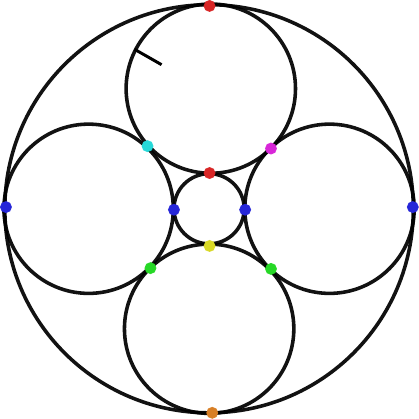}
    \caption{Circle pattern from the incidence graph.\label{fig:chequerboard3}}
  \end{subfigure}\hspace{.2cm}
  \begin{subfigure}[c]{0.5\textwidth}\centering
    \labellist
    \small\hair 2pt
    \pinlabel {$v_2$} [b] at 109 120
    \pinlabel {$u_1$} [b] at 170 120
    \pinlabel {$u_2$} [b] at 49 120
    \pinlabel {$C$} [b] at 204 120
    \endlabellist
    \includegraphics[height=.18\textheight]{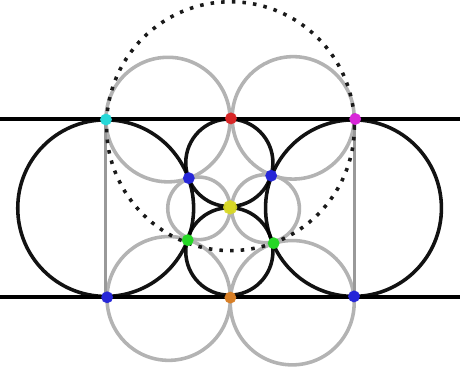}
    \caption{Circle pattern after moving $v_1$ to $\infty$.\label{fig:chequerboard4}}
  \end{subfigure}}
  \caption{The circle pattern arising from the link in \zcref{fig:robot}, illustrating the proof of \zcref{prp:fal}. The red vertices, corresponding to the cusp $ \ast $, will be combined and form the central vertex $v$.\label{fig:chequerboard}}
\end{figure}

\begin{figure}
  \centering
  \labellist
  \small\hair 2pt
  \pinlabel {$v_1$} [b] at 187 437
  \pinlabel {$v_2$} [bl] at 332 345
  \pinlabel {$v_3$} [l] at 379 182
  \pinlabel {$v_4$} [t] at 272 48
  \pinlabel {$u_1$} [bl] at 270 402
  \pinlabel {$u_2$} [l] at 372 272
  \pinlabel {$u_3$} [tl] at 337 109
  \pinlabel {$A_1$} [tr] at 222 319
  \pinlabel {$A_2$} [r] at 271 243
  \pinlabel {$A_3$} [br] at 253 162
  \pinlabel {$C$} [bl] at 28 134
  \endlabellist
  \includegraphics[width=.4\textwidth]{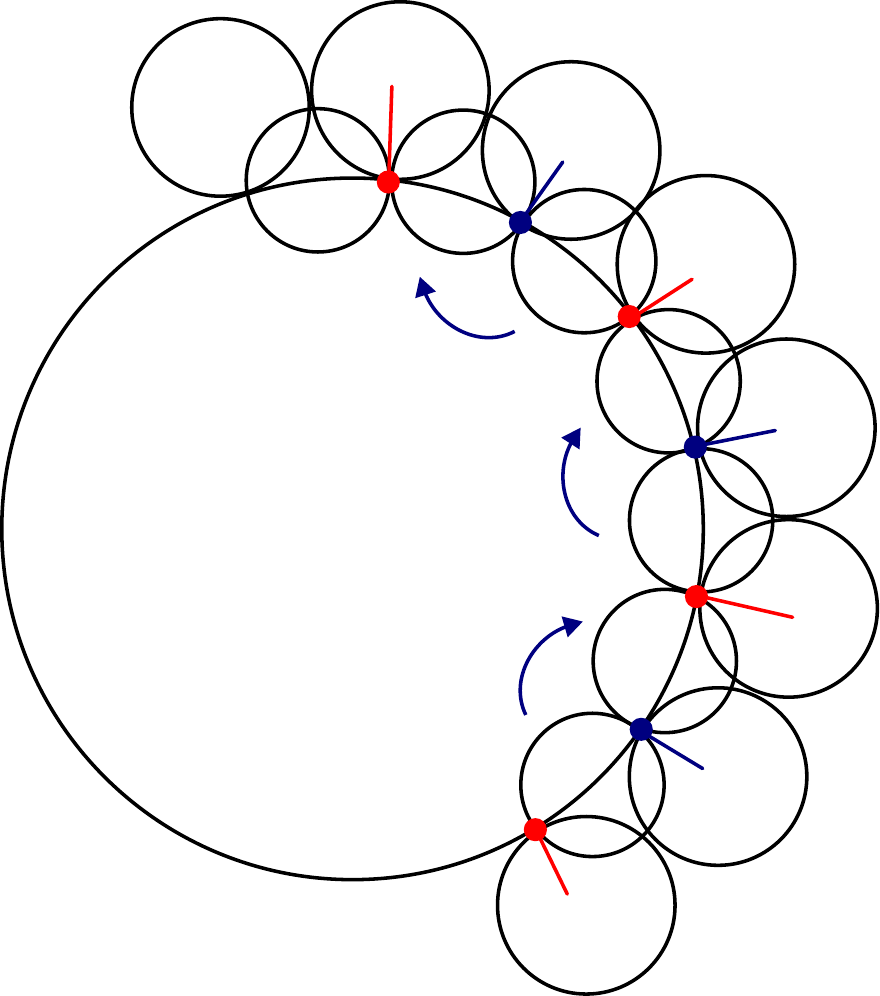}
  \caption{Notation for proof of \zcref{prp:fal}.\label{fig:halo}}
\end{figure}

The proof will be illustrated using the fully augmented link in \zcref{fig:robot_link}.
\begin{proof}
  We will decompose the link $ \mf{l} $ using the so-called \emph{method of chequerboards}; we assume that the reader is somewhat familiar with this procedure,
  which is originally described in the appendix by Agol and Thurston to Lackenby~\cite{lackenby00hal} and which is further explained in \cite{futer07,purcell09,ibarra25,ham23}.

  Let $ \Pi $ be the reflection plane of $ \mf{l} $. Cutting along $ \Pi $ and then cutting along each thrice-punctured sphere transverse to $\Pi $ produces a
  chequerboard pattern as in \zcref{fig:chequerboard1} (compare Figure~2 of \autocite{purcell09}). Contract each triangle into a vertex and interpret each rank $1$
  cusp as an edge. This gives a trivalent graph $\Gamma$ together with an edge-colouring, a map $ \chi : E(\Gamma) \to \mathrm{Cusps}(\mf{l}) $. The conditions
  \ref{cond:fal1} and \ref{cond:fal2} imply that there is a simple edge-cycle $ e_1,\ldots,e_r $ in $ \Gamma $ such that:
  \begin{enumerate}
    \item If $ i $ is odd, then $ \chi(e_i) = \omega $; conversely, if $ \chi(e) = \omega $ for any $ e \in E(\Gamma) $, then $ e = e_i $ for some odd $i$, and
    \item If $ e \in E(\Gamma) $ and $ \chi(e) = \chi(e_i) $ for some $ e_i $, then $ e = e_j $ for some $j$.
  \end{enumerate}
  For instance, in \zcref{fig:chequerboard1} there is a loop with red edges (corresponding to the cusp $ \ast $) alternating with edges of different colours,
  and the total set of colours in this loop appears nowhere else in the edge colouring.

  Take the dual graph of $ \Gamma $ (\zcref{fig:chequerboard2}) and construct the circle pattern $P$ with tangency arrangement given by the graph (\zcref{fig:chequerboard3});
  tangency points inherit the colouring map $\chi$. This circle pattern has a dual pattern $P^\perp$ of orthogonal circles; take the right-angled polyhedron in $ \H^3 $ bounded by the domes above
  the circles in the pattern and above the dual circles, and two copies of this polyhedron glue to give $ \Sph^3 \setminus \mf{l} $: the domes above the dual circles in $ P^\perp $
  are in correspondence with the triangular faces of the chequerboard pattern and glue in pairs to form halves of the discs bounded by augmentation cusps, and the domes
  above our original circles in $ P $ glue across $ \Pi $ onto the corresponding faces of the second copy of the polyhedron.

  From our discussion above, there is a distinguished circle $C$ in $ P $ so that the points of tangency around its circumference alternate between points descending to $ \omega $
  and points which do not. Let $ D $ be one of the two right-angled ideal polyhedra just constructed and pick a vertex of $ D $ that descends to $ \omega $, say $ v_1 $. Then
  label all the vertices of $ D $ around $ C $ as $ v_1, u_1, v_2, u_2, \ldots, v_n, u_n $ where each $ v_i $ descends to $ \omega $. By the discussion of the face-pairing structure on $ D $
  in the previous paragraph, each $ u_i $ is the point of tangency of two faces of $ D $, orthogonal to $C$, paired by a parabolic. Label by $ A_i $ the parabolic pairing the two faces
  above $ u_i $. It follows that, for each $ i $, the map $ A_i \cdots A_1 $ sends $ v_{i+1} $ to $ v_1 $ (see \zcref{fig:halo}). For each $ i > 1 $, there is a circle orthogonal to $C$
  that passes through $ u_{i} $ and $ u_{i+1} $. Cut $D$ along the dome above this circle, and translate the piece under this dome using $ A_i \cdots A_1 $. This procedure, analogous
  to that in \zcref{cons:cut_glue}, rearranges pieces of $ D $ so that it only has a single vertex labelled $ \omega $, at the expense of splitting all the vertices $ u_i $ into pairs.
  Since each $ A_i $ preserves the circle $ C $, the images of the $ u_i $ still lie on the circle $ C $ after this cutting and gluing. In the running example, only one cut and glue need
  be done, along the dark dotted circle in \zcref{fig:chequerboard4}, to produce a polyhedron realising the cusp picture in \zcref{fig:robot_cusps_red} (after doing the inversion across
  the dotted circle, the pattern of coloured vertices of \zcref{fig:chequerboard4} becomes a subset of the dot pattern of \zcref{fig:robot_cusps_red} but rotated by 90 degrees).

  After this procedure has been done, we have produced a new right-angled ideal polyhedron, say $ D^* $. Glue $ D^* $ to a copy $\overline{D^*}$ of itself along the
  face supported on the dome above $ C $. The result is a fundamental polyhedron $D^* \union \overline{D^*}$ for $ \Sph^3 \setminus \mf{l} $, such that the lift
  of $ \omega $ is exactly one ideal vertex $v$. A concave lens can be constructed inside this polyhedron. Indeed, the geodesic plane formed by the dome above $C$
  contains $ v $ and the ideal vertices on the dome are only glued to other vertices on the dome by the side-pairing group; above each of these vertices there is an
  ideal octahedron, obtained by intersecting a horoball neighbourhood of each vertex on $ C $ with the faces of $ D^* \union \overline{D^*} $ and coning the results to $ v $.

  That this concave lens satisfies the conditions (1)--(3) of \zcref{thm:concave_lens} follows from the observations earlier in the proof about the ideal vertices on $ C $. Condition (4)
  is satisfied since all of the polyhedra are symmetric across the dome above $ C $ (since the whole fundamental domain is defined by such a reflection).
\end{proof}

\begin{ex}
  We consider the remaining planar cusps of \zcref{fig:robot}. The cusp $ \dagger $, \zcref{fig:robot_cusps_blue}, does not satisfy the hypotheses of \zcref{thm:concave_lens} as both punctured discs it bounds
  are punctured by longitudes of other planar cusps. The cusp $ \ddagger $, \zcref{fig:robot_cusps_green}, bounds exactly one punctured disc meeting the criteria and hence does satisfy the hypotheses of the theorem.
\end{ex}

\begin{rem}\label{rem:fal_longit}
  By adapting the argument used to prove \zcref{prp:fal}, one can show a dual result. Let $ \mf{l} $ be a fully augmented link $ \mf{l} $ with a planar cusp $ \omega_0 $
  such that:
  \begin{enumerate}[label={(\roman*)}]
    \item There are no self-augmentations of $ \omega_0 $, and
    \item One of the two discs, $D$, bounded by $ \omega_0 $ must be punctured only longitudinally by non-self-augmented planar cusps.
  \end{enumerate}
  Then $\omega_0$ forms the centre of a concave lens satisfying the conditions of \zcref{thm:concave_lens}, and increasing the meridional length of $\omega_0$
  forms cone arcs in the geodesic surface $D$ joining the augmentation cusp punctures together.

  To prove this, consider the edge-coloured graph $\Gamma$ as in the proof of \zcref{prp:fal}. Let $ \omega_1,\ldots,\omega_k $ be the planar
  cusps puncturing $D$. The assumption (ii) implies that there is an edge-cycle in $ \Gamma $ for each $ i $ that is coloured alternately by $ \omega_i $
  and by augmentation cusps. Take the dual circle pattern $ P $; there is a circle $ C_i $ in $ P $ for each $ \omega_i $, as well as a circle $ C_0 $ for
  the disc bounded by $ \omega_0 $ that is \textit{not} $D$. All these circles $ C_i $ for $ 0 \leq i \leq k $ are tangent to a circle $C$ corresponding to the
  disc $D$. Using products of side-pairing parabolics in each circle $ C_i $, we may cut and glue all the $\omega_i$-tangency points on the circles onto the
  single $ \omega_i$-point tangent to the circle $C $ (i.e.\ we do the `glue by $ A_i \cdots A_1 $' step in the proof of \zcref{prp:fal} for each circle $ C_i $).
  Since all of these side-pairings preserve the circle $C_i $, they do not introduce new vertices onto the circle $C$ (it is only moving three vertices, and two
  of those get moved onto $ C_i $ away from the tangency point). Thus, after gluing, we obtain a fundamental polyhedron such that the face supported on the circle $ C $
  meets ideal vertices corresponding to $ \omega_i $ for $ 0 \leq i \leq k $, and no other ideal vertices of polyhedron descend to $\omega_i$. This is enough
  to get a concave lens centred at $ \omega_0 $ as described in the final paragraph of the proof of \zcref{prp:fal}.
\end{rem}

\begin{rem}\label{rem:hk_fal}
  As with the lantern manifolds and stacked Borromean rings, every fully augmented link admits an order $2$ symmetry which cuts the angle along the proposed
  cone axis in half. However, this symmetry is a reflection and not a rotation of angle $ \pi $. Thus the alternative argument in \zcref{rem:hk_deformation}
  via Wei\ss' theorem does not apply as, after quotienting, the cone arc is embedded in a mirror and so the local isotropy groups are not orientation-preserving.
\end{rem}

\subsection{Further examples}
We end the paper with a final string of examples showing interesting phenomena that we do not have space to explore at length. In each case one can read off
a relevant concave lens from the provided cusp picture, as in \zcref{rem:lens_pictures}.

\begin{figure}
  \begin{subfigure}[c]{0.3\textwidth}
    \centering
    \labellist
    \small\hair 2pt
    \pinlabel {$\ast$} [r] at 25 263
    \endlabellist
    \includegraphics[width=\textwidth]{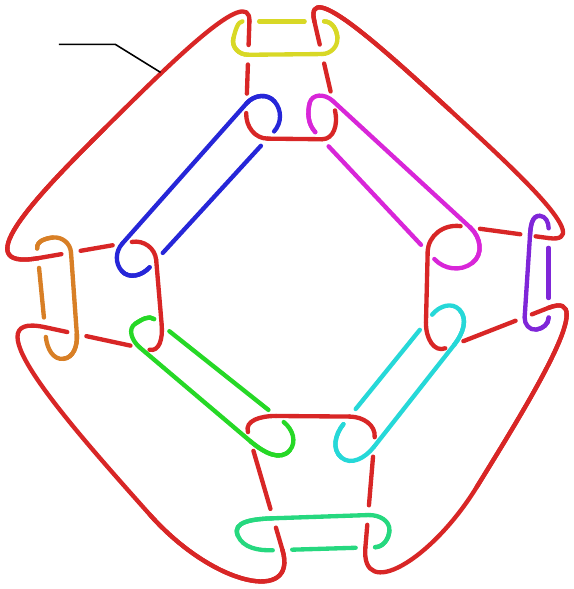}
    \caption{Link in $ \Sph^3$.\label{fig:loebell_link}}
  \end{subfigure}\\[.2cm]
  \begin{subfigure}[c]{\textwidth}
    \includegraphics[width=\textwidth]{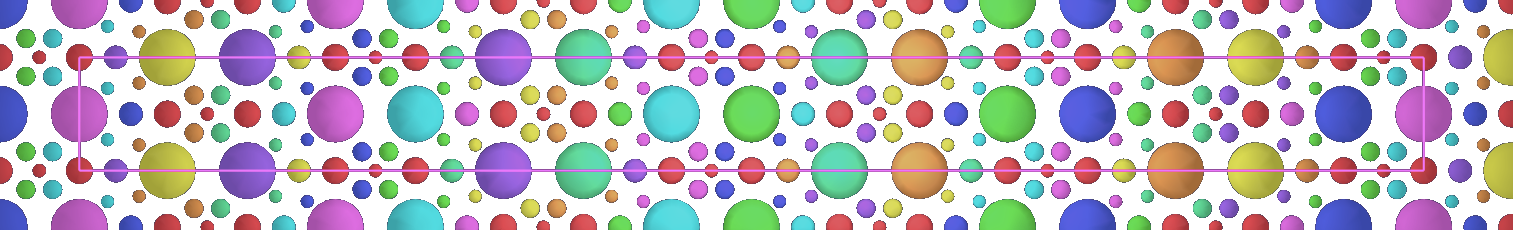}
    \caption{Cusp horoballs with $\ast$ at $ \infty $.\label{fig:loebell_cusps}}
  \end{subfigure}
  \caption{The L\"obell link $L(4)$.\label{fig:loebell}}
\end{figure}
\begin{ex}\label{ex:loebell}
  This example will show that the non-self-augmentation condition of \zcref{prp:fal} cannot be removed. In \zcref{fig:loebell} we show the L\"obell
  link $ L(4) $, see Chesebro, DeBlois, and Wilton~\cite[\S7.2]{chesebro12}. Since the link is fully augmented, there is an embedded geodesic surface spanned by the long planar cusp $ \ast $.
  However, it meets that cusp multiple times and as a consequence there are tetrahedra in the triangulation produced by the algorithm in the proof of \zcref{prp:fal}
  that contain two vertices descending to $\ast$; it follows that there is no concave lens in this triangulation which has a central cusp corresponding to $ \ast $.
  One can see in the horoball diagram there are vertices (in red) descending to $ \ast $; following the proof of \zcref{thm:concave_lens}, the vertex at infinity
  corresponding to $\ast$ would remain parabolic while the other (equivalent) vertices would be pulled into ideal quadrilaterals, completely breaking the side-pairing
  combinatorics. In other words, no version of our cone deformations can be done along $\ast$.
\end{ex}

\begin{figure}
  \begin{subfigure}[c]{0.4\textwidth}
    \includegraphics[width=\textwidth]{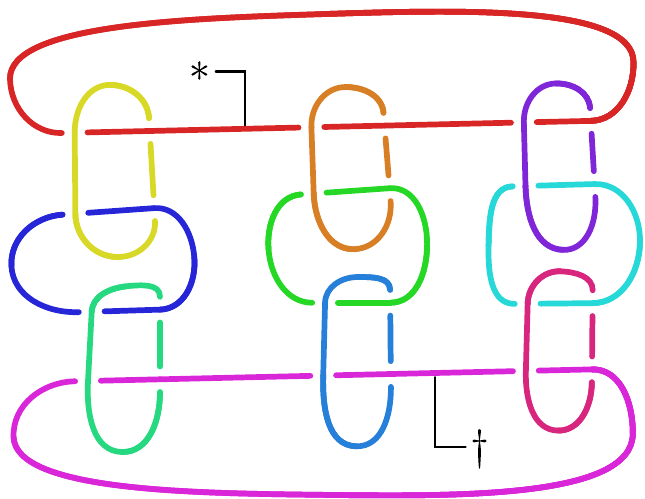}
    \caption{Link in $ \Sph^3 $.\label{fig:11cpt_link}}
  \end{subfigure}\hfill
  \begin{subfigure}[c]{0.55\textwidth}
    \includegraphics[width=\textwidth]{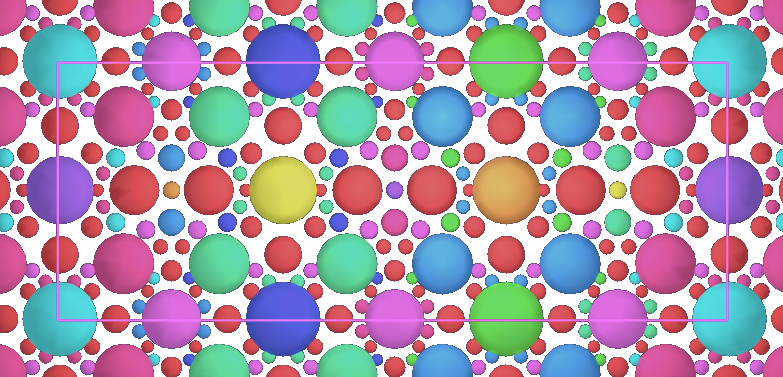}
    \caption{Cusp horoballs with $\ast$ at $ \infty $.\label{fig:11cpt_cusps}}
  \end{subfigure}
  \caption{The augmented link from \zcref{ex:11cpt}.\label{fig:11cpt}}
\end{figure}

\begin{ex}\label{ex:11cpt}
  We now show a fully augmented link with two distinct cone-deformations centred at the same cusp. Consider the link of $11$ components drawn in \zcref{fig:11cpt_link}.
  \begin{itemize}
    \item The surface spanned by the cusp $ \ast $ which is punctured thrice by augmentation cusps is the centre of a concave lens satisfying the conditions of \zcref{thm:concave_lens}.
          It can be deformed to produce six thrice-punctured spheres in total, two from each augmentation cusp. It is this lens which is centred in \zcref{fig:11cpt_cusps}.
    \item The other surface bounded by $ \ast $ in the reflection plane of the link is punctured by four longitudes and is the centre of a concave lens satisfying the conditions of \zcref{rem:fal_longit};
          this lens is visible along the top and bottom horizontal edges of the rectangle in \zcref{fig:11cpt_cusps}.
          It can be deformed to obtain three new cone arcs so that, in the limit, the cusp $ \dagger $ splits into three thrice-punctured spheres.
  \end{itemize}
\end{ex}

\begin{figure}
  \begin{subfigure}[c]{0.4\textwidth}
    \includegraphics[width=\textwidth]{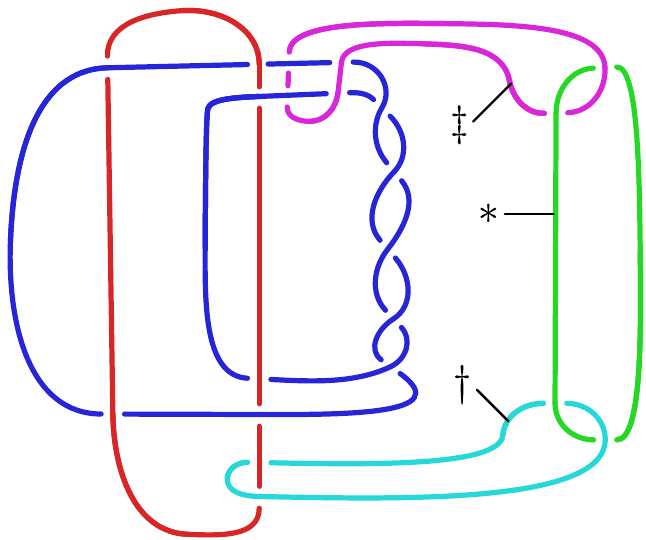}
    \caption{Link in $ \Sph^3$.\label{fig:l11a202mod_link}}
  \end{subfigure}\hfill
  \begin{subfigure}[c]{0.49\textwidth}
    \includegraphics[width=\textwidth]{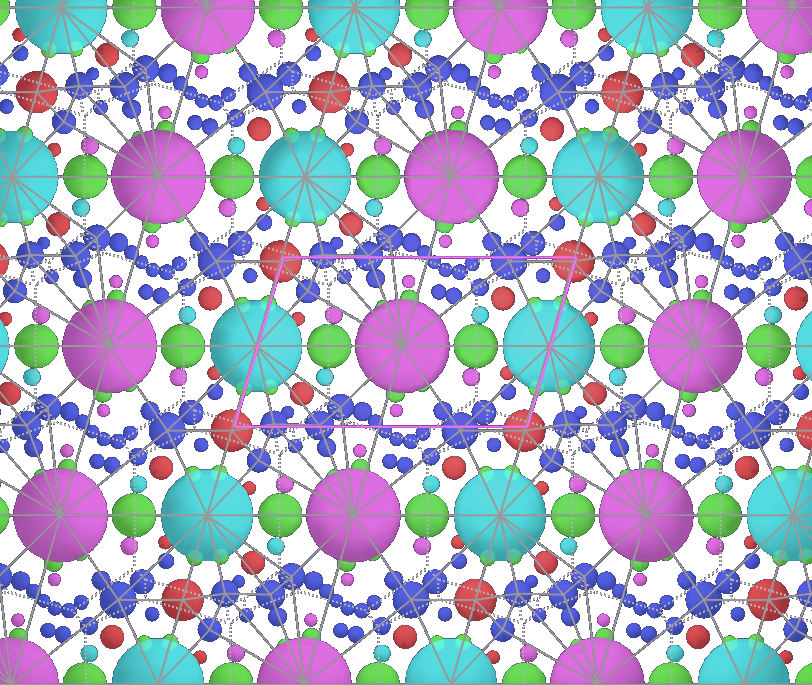}
    \caption{Cusp horoballs with $\ast$ at $ \infty $.\label{fig:l11a202mod_cusp_1}}
  \end{subfigure}\\[.5em]
  \begin{subfigure}[c]{0.49\textwidth}
    \includegraphics[width=\textwidth]{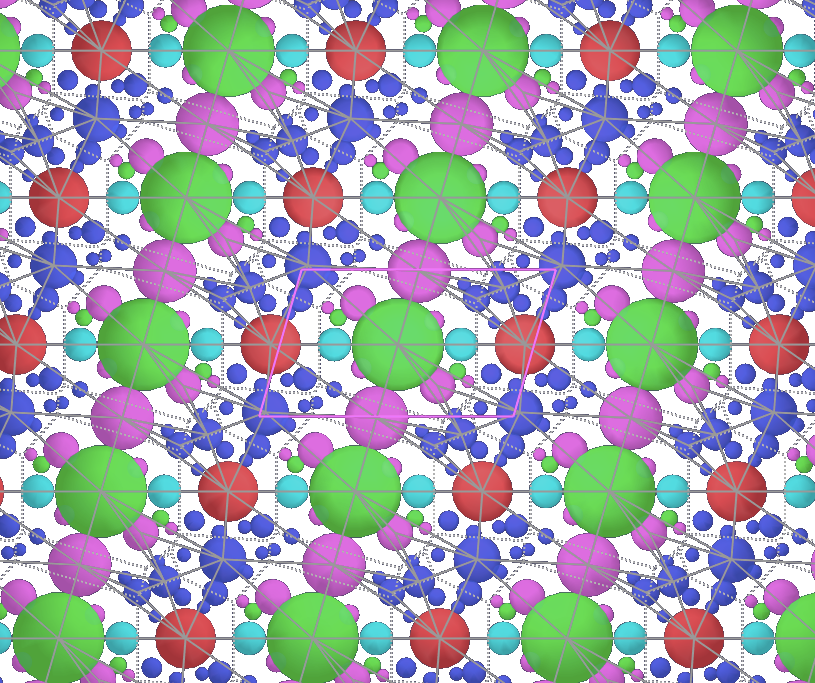}
    \caption{Cusp horoballs with $\dagger$ at $ \infty $.\label{fig:l11a202mod_cusp_2}}
  \end{subfigure}\hfill
  \begin{subfigure}[c]{0.49\textwidth}
    \includegraphics[width=\textwidth]{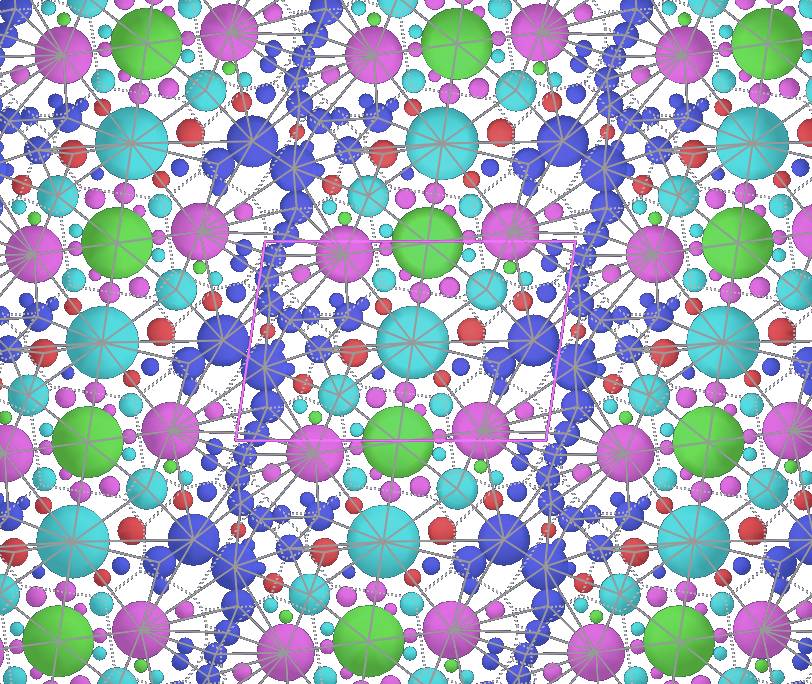}
    \caption{Cusp horoballs with $\ddagger$ at $ \infty $.\label{fig:l11a202mod_cusp_3}}
  \end{subfigure}
  \caption{A link obtained from \href{https://katlas.org/wiki/L11a202}{\texttt{L11a202}}.\label{fig:l11a202mod}}
\end{figure}
\begin{ex}
  We consider now an example of a link with twist regions, to show that it is not enough for the cusp being elongated to be untwisted.
  In \zcref{fig:l11a202mod} we show a link obtained by adding a chain of three cusps to \href{https://katlas.org/wiki/L11a202}{\texttt{L11a202}}. We can
  see cusp horoballs which give the outer vertices of a concave lens in \zcref{fig:l11a202mod_cusp_1,fig:l11a202mod_cusp_2} but not in \zcref{fig:l11a202mod_cusp_3},
  even though all of the marked cusps bound embedded discs punctured only by meridians.
\end{ex}

\appendix
\section{The polyhedron theorem for cone manifolds}\label{sec:poincare}
We describe here a generalisation for cone manifold groups of the Poincar\'e polyhedron theorem which can be proved by following the usual proof found, for example, in Maskit~\cite[\S IV.H]{maskit}.
It is certainly well-known to experts, and was used implicitly in Elzenaar~\cite{elzenaar24c}; its advantage in studying cone manifolds over the usual techniques via Thurston's gluing equations is
that it gives an explicit presentation for the holonomy group. For further discussion on different polyhedron theorems see Epstein and Petronio~\cite{epstein94}.

Let $ \mathbb{X} $ be one of $ \H^n $, $ \mathbb{E}^n $, or $ \Sph^n $ where $ n \geq 2 $ and let $ \mathbb{G} $ be the isometry group of $ \mathbb{X} $.
Suppose that $ D \subset \mathbb{X} $ is an open polyhedron, i.e.\ a set that is locally an intersection of half-spaces in $ \mathbb{X} $. Assume also that $ D $ has
countably many faces, and is locally finite, i.e.\ the link of each vertex is a finite-sided polyhedron. For every facet $s$ (codimension $1$ face) of $ D $, suppose
that there is a facet $ s' $ of $ D $ (not necessarily distinct from $s$) and an element $ g_s \in \mathbb{G} $ satisfying the following conditions:
\begin{enumerate}[label={(\roman*)}]
  \item $ g_s(s) = s' $.\label{item:ss}
  \item $ g_{s'} = g_s^{-1} $. \label{item:rel1}
  \item There is an open neighbourhood $U$ of $s$ in $D$ such that $ g_s(U) \inter D = \emptyset $.
\end{enumerate}
The side-pairing structure thus defined induces an equivalence relation on $ \overline{D} $ (in the case $ \mathbb{X} = \H^n $, $ \overline{D} $ also includes the boundary
at infinity of $ D $) which is the identity on $ D $.  Let $ D^* $ be the quotient of $ \overline{D} $ by this relation with projection map $ p : \overline{D} \to D^* $.
\begin{enumerate}[label={(\roman*)},resume]
  \item For every $ z \in D^* $, $ p^{-1}(z) $ is a finite set.\label{item:is_finite}
\end{enumerate}

A \df{generalised edge} of $ D $, henceforth just \df{edge}, is either a codimension $2$ face, or a point of tangency of two faces of $ D $ (meaningful only in the case $ \mathbb{X} = \H^n $, where
it behaves as an `edge of length $0$' with dihedral angle $0$). For such an edge $e$, pick arbitrarily a facet $ s_1 $ incident to $ e_1 = e $. Then $ s_1' $ is adjacent to the translate $ g_{s_1}(e_1) $. Now for
each $ i $ let $ s_i $ be the facet incident to $ e_i $ that is distinct from $ s'_{i-1} $. Let
\begin{displaymath}
  n = \min \{ i \in \N : s_i = s_1\; \text{and}\; i > 1 \};
\end{displaymath}
this minimum exists by \ref{item:is_finite}. The \df{edge cycle} of $ e $ is $ \{ e_1,\ldots,e_{n-1} \} $. The \df{cycle transformation} $ h_e $, is by definition, $ g_{s_{n-1}} \cdots g_{s_1} $.
For each edge $ e $ of $ D $ let $ \alpha(e) $ be the dihedral angle of $D$ across $e$ and set
\begin{displaymath}
  \theta(e) = \sum_{m=1}^n \alpha(e_m)
\end{displaymath}
where $ \{e_1,\ldots,e_n\} $ is the edge cycle of $ e $.
\begin{enumerate}[label={(\roman*)},resume]
  \item The element $ h_e \in \mathbb{G} $ is a rotation through an angle $ \theta(e) $ if $ \theta(e) > 0 $. If $ \theta(e) = 0 $ then
        necessarily $ \mathbb{X} = \mathbb{H}^n $; in this case $ h_e $ must be parabolic.\label{item:rel2}
\end{enumerate}

\begin{thm}
  Suppose that $ D \subset \mathbb{C} $ is a polyhedron with side-pairing maps $ (g_s) $ satisfying conditions \ref{item:ss}--\ref{item:rel2} above. Let $ G = \langle g_s : s\;\text{a side of}\;D \rangle $.
  Then $ D^* $ admits a cone manifold structure locally modelled on $ \mathbb{X} $, the singular locus of $ D^* $ is the projection of the set of edges $e$ of $ D $ with $ \theta(e) \not\in \{0,2\pi\} $,
  and $ G = \Hol(D^*) $. The relations induced by \ref{item:rel1} and \ref{item:rel2} give a complete set of relations for $ G$.
\end{thm}
\begin{proof}
  One follows exactly the proof in Maskit~\cite[\S IV.H]{maskit}, except we have replaced the condition on completeness with the parabolic cycle condition of \ref{item:rel2}. To
  see that the latter condition gives completeness of the path metric, note that the manifold is trivially complete everywhere except vertices on the sphere at infinity. It is here that we use
  that $ D $ is locally finite---if a sequence that is Cauchy in the hyperbolic metric is converging to such a point then it must eventually enter a horoball around it. We may assume the link of
  the polyhedron that is cut out by this horoball is finite, and the proof of Proposition~IV.I.6 of Maskit~\cite{maskit} goes through.
\end{proof}

As mentioned in Elzenaar~\cite{elzenaar24c}, versions of the Maskit combination theorems also go through for cone manifolds since they depend
only on cutting and gluing fundamental domains and are just explicit versions of gluing results for general $ \mathrm{CAT}(0) $ spaces: for instance,
the first Maskit combination theorem is essentially found as Theorem~II.11.18 of Bridson and Haefliger~\cite{bridson_haefliger}, and the second as Proposition~II.11.21 \textit{op. cit}.
Both theorems just cited are global results on $ \mathrm{CAT}(0) $ spaces. To recover the corresponding polyhedron theorems, one takes the fundamental polyhedra
for the component groups being glued, produces a $ \mathrm{CAT}(0) $ space from each by tiling with the side-pairings (one can view this as an infinitely branched and highly singular cover
of $ \H^3 $), and then uses the global results on this tiled space---the proofs in \cite{bridson_haefliger} go by immediately slicing the space back up into fundamental polyhedra,
gluing those together locally, and then developing the result to obtain a  $ \mathrm{CAT}(0) $ space on which the amalgamated group acts.

\sloppy
\printbibliography

@article{boileau05,
 author = {Boileau, Michel and Leeb, Bernhard and Porti, Joan},
 title = {Geometrization of $3$-dimensional orbifolds},
 fjournal = {Annals of Mathematics. Second Series},
 journal = {Ann. Math. (2)},
 issn = {0003-486X},
 volume = {162},
 number = {1},
 pages = {195--290},
 year = {2005},
 language = {English},
 doi = {10.4007/annals.2005.162.195},
 zbMATH = {5011495},
 Zbl = {1087.57009}
}

@book{thurstonN,
 author = {Thurston, William P.},
 editor = {Kerckhoff, Steven P. and Farb, Benson and Gabai, David},
 title = {Collected works of {William} {P}. {Thurston} with commentary: {IV}. {The} geometry and topology of three-manifolds},
 note = {With a preface by {Steven} {P}. {Kerckhoff}},
 isbn = {978-1-4704-6391-5; 978-1-4704-6836-1; 978-1-4704-5164-6},
 year = {2022},
 publisher = {Providence, RI: American Mathematical Society (AMS)},
 language = {English},
 zbMATH = {7581723},
 Zbl = {1507.57005}
}

@article{thurston98,
 author = {Thurston, William P.},
 title = {How to see {{\(3\)}}-manifolds},
 fjournal = {Classical and Quantum Gravity},
 journal = {Classical Quantum Gravity},
 issn = {0264-9381},
 volume = {15},
 number = {9},
 pages = {2545--2571},
 year = {1998},
 language = {English},
 doi = {10.1088/0264-9381/15/9/004},
 zbMATH = {1364007},
 Zbl = {0932.57017}
}

@article{brock04,
 author = {Brock, Jeffrey F. and Bromberg, Kenneth W.},
 title = {On the density of geometrically finite {Kleinian} groups.},
 fjournal = {Acta Mathematica},
 journal = {Acta Math.},
 issn = {0001-5962},
 volume = {192},
 number = {1},
 pages = {33--93},
 year = {2004},
 language = {English},
 doi = {10.1007/BF02441085},
 zbMATH = {2115792},
 Zbl = {1055.57020}
}

@article{bromberg04,
 author = {Bromberg, K.},
 title = {Hyperbolic cone-manifolds, short geodesics, and {Schwarzian} derivatives},
 fjournal = {Journal of the American Mathematical Society},
 journal = {J. Am. Math. Soc.},
 issn = {0894-0347},
 volume = {17},
 number = {4},
 pages = {783--826},
 year = {2004},
 language = {English},
 doi = {10.1090/S0894-0347-04-00462-X},
 zbMATH = {2106876},
 Zbl = {1061.30037}
}

@article{hk05,
 author = {Hodgson, Craig D. and Kerckhoff, Steven P.},
 title = {Universal bounds for hyperbolic {Dehn} surgery},
 fjournal = {Annals of Mathematics. Second Series},
 journal = {Ann. Math. (2)},
 issn = {0003-486X},
 volume = {162},
 number = {1},
 pages = {367--421},
 year = {2005},
 language = {English},
 doi = {10.4007/annals.2005.162.367},
 zbMATH = {5011497},
 Zbl = {1087.57011}
}

@article{futer22b,
 author = {Futer, David and Purcell, Jessica S. and Schleimer, Saul},
 title = {Effective drilling and filling of tame hyperbolic $3$-manifolds},
 fjournal = {Commentarii Mathematici Helvetici},
 journal = {Comment. Math. Helv.},
 issn = {0010-2571},
 volume = {97},
 number = {3},
 pages = {457--512},
 year = {2022},
 language = {English},
 doi = {10.4171/CMH/536},
 zbMATH = {7574046},
 Zbl = {1505.57026}
}

@article{futer22,
 author = {Futer, David and Purcell, Jessica S. and Schleimer, Saul},
 title = {Effective bilipschitz bounds on drilling and filling},
 fjournal = {Geometry \& Topology},
 journal = {Geom. Topol.},
 issn = {1465-3060},
 volume = {26},
 number = {3},
 pages = {1077--1188},
 year = {2022},
 language = {English},
 doi = {10.2140/gt.2022.26.1077},
 zbMATH = {7584041},
 Zbl = {1502.30126}
}

@article{hk98,
 author = {Hodgson, Craig D. and Kerckhoff, Steven P.},
 title = {Rigidity of hyperbolic cone-manifolds and hyperbolic {Dehn} surgery},
 fjournal = {Journal of Differential Geometry},
 journal = {J. Differ. Geom.},
 issn = {0022-040X},
 volume = {48},
 number = {1},
 pages = {1--59},
 year = {1998},
 language = {English},
 doi = {10.4310/jdg/1214460606},
 zbMATH = {1187443},
 Zbl = {0919.57009}
}

@article{weiss05,
 author = {Wei\ss, Hartmut},
 title = {Local rigidity of $3$-dimensional cone-manifolds},
 fjournal = {Journal of Differential Geometry},
 journal = {J. Differ. Geom.},
 issn = {0022-040X},
 volume = {71},
 number = {3},
 pages = {437--506},
 year = {2005},
 language = {English},
 doi = {10.4310/jdg/1143571990},
 zbMATH = {5033805},
 Zbl = {1098.53038}
}

@article{weiss07,
 author = {Wei\ss, Hartmut},
 title = {Global rigidity of $3$-dimensional cone-manifolds},
 fjournal = {Journal of Differential Geometry},
 journal = {J. Differ. Geom.},
 issn = {0022-040X},
 volume = {76},
 number = {3},
 pages = {495--523},
 year = {2007},
 language = {English},
 doi = {10.4310/jdg/1180135696},
 zbMATH = {5174441},
 Zbl = {1184.53049}
}

@article{kojima98,
 author = {Kojima, Sadayoshi},
 title = {Deformations of hyperbolic {{\(3\)}}-cone-manifolds},
 fjournal = {Journal of Differential Geometry},
 journal = {J. Differ. Geom.},
 issn = {0022-040X},
 volume = {49},
 number = {3},
 pages = {469--516},
 year = {1998},
 language = {English},
 doi = {10.4310/jdg/1214461108},
 zbMATH = {1725383},
 Zbl = {0990.57004}
}

@article{weiss13,
 author = {Wei{\ss}, Hartmut},
 title = {The deformation theory of hyperbolic cone-$3$-manifolds with cone-angles less than {{\(2\pi\)}}},
 fjournal = {Geometry \& Topology},
 journal = {Geom. Topol.},
 issn = {1465-3060},
 volume = {17},
 number = {1},
 pages = {329--367},
 year = {2013},
 language = {English},
 doi = {10.2140/gt.2013.17.329},
 zbMATH = {6152265},
 Zbl = {1262.53032}
}

@article{montcouquiol13,
 author = {Montcouquiol, Gr{\'e}goire},
 title = {Deformations of hyperbolic convex polyhedra and cone-$3$-manifolds},
 fjournal = {Geometriae Dedicata},
 journal = {Geom. Dedicata},
 issn = {0046-5755},
 volume = {166},
 pages = {163--183},
 year = {2013},
 language = {English},
 doi = {10.1007/s10711-012-9790-5},
 zbMATH = {6222296},
 Zbl = {1279.52015}
}

@THESIS{hodgson,
  AUTHOR = {Hodgson, Craig D.},
  DATE = {1986},
  INSTITUTION = {Princeton University},
  TITLE = {Degeneration and regeneration of geometric structures on three-manifolds},
  TYPE = {Doctoral thesis}
}

@article{porti02,
 author = {Porti, Joan},
 title = {Regenerating hyperbolic cone structures from {Nil}},
 fjournal = {Geometry \& Topology},
 journal = {Geom. Topol.},
 issn = {1465-3060},
 volume = {6},
 pages = {815--852},
 year = {2002},
 language = {English},
 doi = {10.2140/gt.2002.6.815},
 url = {https://eudml.org/doc/123317},
 zbMATH = {2017788},
 Zbl = {1032.57015}
}

@article{cooper18,
 author = {Cooper, Daryl and Danciger, Jeffrey and Wienhard, Anna},
 title = {Limits of geometries},
 fjournal = {Transactions of the American Mathematical Society},
 journal = {Trans. Am. Math. Soc.},
 issn = {0002-9947},
 volume = {370},
 number = {9},
 pages = {6585--6627},
 year = {2018},
 language = {English},
 doi = {10.1090/tran/7174},
 zbMATH = {6891720},
 Zbl = {1395.57022}
}

@article{heusener01,
 author = {Heusener, Michael and Porti, Joan and Su{\'a}rez, Eva},
 title = {Regenerating singular hyperbolic structures from {Sol}},
 fjournal = {Journal of Differential Geometry},
 journal = {J. Differ. Geom.},
 issn = {0022-040X},
 volume = {59},
 number = {3},
 pages = {439--478},
 year = {2001},
 language = {English},
 doi = {10.4310/jdg/1090349448},
 zbMATH = {2053205},
 Zbl = {1042.57008}
}

@article{porti98,
 author = {Porti, Joan},
 title = {Regenerating hyperbolic and spherical cone structures from {Euclidean} ones},
 fjournal = {Topology},
 journal = {Topology},
 issn = {0040-9383},
 volume = {37},
 number = {2},
 pages = {365--392},
 year = {1998},
 language = {English},
 doi = {10.1016/S0040-9383(97)00025-6},
 zbMATH = {1113440},
 Zbl = {0897.58042}
}

@article{akiyoshi18,
 author = {Akiyoshi, Hirotaka},
 title = {Thin representations for the one-cone torus group},
 fjournal = {Topology and its Applications},
 journal = {Topology Appl.},
 issn = {0166-8641},
 volume = {264},
 pages = {115--144},
 year = {2019},
 language = {English},
 doi = {10.1016/j.topol.2019.06.025},
 zbMATH = {7088362},
 Zbl = {1422.51006}
}

@article{yoshida22,
 author = {Yoshida, Ken'ichi},
 title = {Degeneration of $3$-dimensional hyperbolic cone structures with decreasing cone angles},
 fjournal = {Conformal Geometry and Dynamics},
 journal = {Conform. Geom. Dyn.},
 issn = {1088-4173},
 volume = {26},
 pages = {182--193},
 year = {2022},
 language = {English},
 doi = {10.1090/ecgd/375},
 zbMATH = {7615749},
 Zbl = {1514.57031}
}

@MISC{elzenaar24c,
  AUTHOR = {Elzenaar, Alex},
  DATE = {2024},
  EPRINT = {2411.17940},
  EPRINTCLASS = {math.GT},
  EPRINTTYPE = {arXiv},
  TITLE = {Changing topological type of compression bodies through cone manifolds},
}

@article{neumann91,
 author = {Neumann, Walter D. and Reid, Alan W.},
 title = {Amalgamation and the invariant trace field of a {Kleinian} group},
 fjournal = {Mathematical Proceedings of the Cambridge Philosophical Society},
 journal = {Math. Proc. Camb. Philos. Soc.},
 issn = {0305-0041},
 volume = {109},
 number = {3},
 pages = {509--515},
 year = {1991},
 language = {English},
 doi = {10.1017/S0305004100069942},
 zbMATH = {4203280},
 Zbl = {0728.57009}
}

@article{neumann93,
 author = {Neumann, Walter D. and Reid, Alan W.},
 title = {Rigidity of cusps in deformations of hyperbolic $3$-orbifolds},
 fjournal = {Mathematische Annalen},
 journal = {Math. Ann.},
 issn = {0025-5831},
 volume = {295},
 number = {2},
 pages = {223--237},
 year = {1993},
 language = {English},
 doi = {10.1007/BF01444885},
 url = {https://eudml.org/doc/165040},
 zbMATH = {537250},
 Zbl = {0813.57013}
}

@MISC{kapovich92,
  AUTHOR = {Kapovich, Michael},
  URL = {https://www.math.ucdavis.edu/~kapovich/EPR/eis.pdf},
  DATE = {1992},
  HOWPUBLISHED = {MSRI Preprint},
  TITLE = {Eisenstein series and Dehn surgery},
}

@article{calegari96,
 author = {Calegari, Danny},
 title = {A note on strong geometric isolation in $3$-orbifolds},
 fjournal = {Bulletin of the Australian Mathematical Society},
 journal = {Bull. Aust. Math. Soc.},
 issn = {0004-9727},
 volume = {53},
 number = {2},
 pages = {271--280},
 year = {1996},
 language = {English},
 doi = {10.1017/S0004972700016993},
 zbMATH = {888865},
 Zbl = {0853.57010}
}

@article{calegari01,
 author = {Calegari, Danny},
 title = {Napoleon in isolation},
 fjournal = {Proceedings of the American Mathematical Society},
 journal = {Proc. Am. Math. Soc.},
 issn = {0002-9939},
 volume = {129},
 number = {10},
 pages = {3109--3119},
 year = {2001},
 language = {English},
 doi = {10.1090/S0002-9939-01-05915-9},
 zbMATH = {1614781},
 Zbl = {0971.57022}
}

@article{purcell08,
 author = {Purcell, Jessica S.},
 title = {Cusp shapes under cone deformation},
 fjournal = {Journal of Differential Geometry},
 journal = {J. Differ. Geom.},
 issn = {0022-040X},
 volume = {80},
 number = {3},
 pages = {453--500},
 year = {2008},
 language = {English},
 doi = {10.4310/jdg/1226090484},
 zbMATH = {5504165},
 Zbl = {1182.57015}
}

@article{wielenberg78,
 author = {Wielenberg, Norbert J.},
 title = {The structure of certain subgroups of the {Picard} group},
 fjournal = {Mathematical Proceedings of the Cambridge Philosophical Society},
 journal = {Math. Proc. Camb. Philos. Soc.},
 issn = {0305-0041},
 volume = {84},
 pages = {427--436},
 year = {1978},
 language = {English},
 doi = {10.1017/S0305004100055250},
 zbMATH = {3620720},
 Zbl = {0399.57005}
}

@book{farb,
 author = {Farb, Benson and Margalit, Dan},
 title = {A primer on mapping class groups},
 fseries = {Princeton Mathematical Series},
 series = {Princeton Math. Ser.},
 volume = {49},
 isbn = {978-0-691-14794-9; 978-1-400-83904-9},
 year = {2011},
 publisher = {Princeton, NJ: Princeton University Press},
 language = {English},
 zbMATH = {5960418},
 Zbl = {1245.57002}
}

@misc{SnapPy32,
     author={Culler, Marc and Dunfield, Nathan M. and Goerner,
     Matthias and Weeks, Jeffrey R.},
     title={\textsc{Snap{P}y} 3.2, a computer program for studying the geometry and topology of $3$-manifolds},
     howpublished={Available at \url{http://snappy.computop.org}.}
}

@article{ibarra25,
 author = {Ibarra, Dionne and McQuire, Emma N. and Purcell, Jessica S.},
 title = {Augmented links, shadow links, and the {TV} volume conjecture: a geometric perspective},
 fjournal = {The New York Journal of Mathematics},
 journal = {New York J. Math.},
 issn = {1076-9803},
 volume = {32},
 pages = {1--34},
 year = {2026},
 language = {English},
 url = {https://nyjm.albany.edu/j/2026/32-1.html},
 zbMATH = {8158543}
}

@article{pinsky23,
 author = {Pinsky, Tali and Purcell, Jessica S. and Rodr{\'{\i}}guez-Migueles, Jos{\'e} Andr{\'e}s},
 title = {Arithmetic modular links},
 fjournal = {Pacific Journal of Mathematics},
 journal = {Pac. J. Math.},
 issn = {1945-5844},
 volume = {327},
 number = {2},
 pages = {337--358},
 year = {2023},
 language = {English},
 doi = {10.2140/pjm.2023.327.337},
 zbMATH = {7818425},
 Zbl = {1539.57016}
}

@incollection{purcell09,
 author = {Purcell, Jessica S.},
 title = {An introduction to fully augmented links},
 booktitle = {Interactions between hyperbolic geometry, quantum topology and number theory. Proceedings of a workshop, June 3--13, 2009 and a conference, June 15--19, 2009, Columbia University, New York, NY, USA},
 isbn = {978-0-8218-4960-6},
 pages = {205--220},
 year = {2011},
 publisher = {Providence, RI: American Mathematical Society (AMS)},
 language = {English},
 zbMATH = {5954472},
 Zbl = {1236.57006},
 fseries = {Contemporary Mathematics},
 series = {Contemp. Math.},
 editor = {Champagnerkar, Abhijit and Dasbach, Oliver and Kalfagianni, Efstratia and Kofman, Ilya and Neumann, Walter and Stoltzfus, Neal},
 issn = {0271-4132},
 number = {541},
 isbn = {978-0-8218-4960-6},
}

@article{futer07,
 author = {Futer, David and Purcell, Jessica S.},
 title = {Links with no exceptional surgeries},
 fjournal = {Commentarii Mathematici Helvetici},
 journal = {Comment. Math. Helv.},
 issn = {0010-2571},
 volume = {82},
 number = {3},
 pages = {629--664},
 year = {2007},
 language = {English},
 doi = {10.4171/CMH/105},
 zbMATH = {5227547},
 Zbl = {1134.57003}
}

@article{bleiler96,
 author = {Bleiler, Steven A. and Hodgson, Craig D.},
 title = {Spherical space forms and {Dehn} filling},
 fjournal = {Topology},
 journal = {Topology},
 issn = {0040-9383},
 volume = {35},
 number = {3},
 pages = {809--833},
 year = {1996},
 language = {English},
 doi = {10.1016/0040-9383(95)00040-2},
 zbMATH = {912202},
 Zbl = {0863.57009}
}

@book{akiyoshi,
 author = {Akiyoshi, Hirotaka and Sakuma, Makoto and Wada, Masaaki and Yamashita, Yasushi},
 title = {Punctured torus groups and $2$-bridge knot groups. {I}},
 fseries = {Lecture Notes in Mathematics},
 series = {Lect. Notes Math.},
 issn = {0075-8434},
 number = {1909},
 isbn = {978-3-540-71806-2},
 year = {2007},
 publisher = {Berlin: Springer},
 language = {English},
 doi = {10.1007/978-3-540-71807-9},
 zbMATH = {5152430},
 Zbl = {1132.57001}
}

@article{lee13,
 author = {Lee, Donghi and Sakuma, Makoto},
 title = {A variation of {McShane}'s identity for $2$-bridge links},
 fjournal = {Geometry \& Topology},
 journal = {Geom. Topol.},
 issn = {1465-3060},
 volume = {17},
 number = {4},
 pages = {2061--2101},
 year = {2013},
 language = {English},
 doi = {10.2140/gt.2013.17.2061},
 zbMATH = {6198081},
 Zbl = {1311.57022}
}

@article{gueritaud06,
 author = {Gu{\'e}ritaud, Fran{\c{c}}ois},
 title = {On canonical triangulations of once-punctured torus bundles and two-bridge link complements},
 note = {{With} an appendix by {David} {Futer}},
 fjournal = {Geometry \& Topology},
 journal = {Geom. Topol.},
 issn = {1465-3060},
 volume = {10},
 pages = {1239--1284},
 year = {2006},
 language = {English},
 doi = {10.2140/gt.2006.10.1239},
 zbMATH = {5117941},
 Zbl = {1130.57024}
}

@article{lackenby00hal,
 author = {Lackenby, Marc},
 title = {The volume of hyperbolic alternating link complements},
 note = {With an appendix by {Ian} {Agol} and {Dylan} {Thurston}},
 fjournal = {Proceedings of the London Mathematical Society. Third Series},
 journal = {Proc. Lond. Math. Soc. (3)},
 issn = {0024-6115},
 volume = {88},
 number = {1},
 pages = {204--224},
 year = {2004},
 language = {English},
 doi = {10.1112/S0024611503014291},
 zbMATH = {2063095},
 Zbl = {1041.57002}
}

@article{ham23,
 author = {Ham, Sophie L. and Purcell, Jessica S.},
 title = {Geometric triangulations and highly twisted links},
 fjournal = {Algebraic \& Geometric Topology},
 journal = {Algebr. Geom. Topol.},
 issn = {1472-2747},
 volume = {23},
 number = {3},
 pages = {1399--1462},
 year = {2023},
 language = {English},
 doi = {10.2140/agt.2023.23.1399},
 zbMATH = {7706501},
 Zbl = {1529.57002}
}

@article{chesebro12,
 author = {Chesebro, Eric and DeBlois, Jason and Wilton, Henry},
 title = {Some virtually special hyperbolic $3$-manifold groups},
 fjournal = {Commentarii Mathematici Helvetici},
 journal = {Comment. Math. Helv.},
 issn = {0010-2571},
 volume = {87},
 number = {3},
 pages = {727--787},
 year = {2012},
 language = {English},
 doi = {10.4171/CMH/267},
 zbMATH = {6067464},
 Zbl = {1283.57007}
}

@book{maskit,
 author = {Maskit, Bernard},
 title = {Kleinian groups},
 fseries = {Grundlehren der Mathematischen Wissenschaften},
 series = {Grundlehren Math. Wiss.},
 issn = {0072-7830},
 number = {287},
 isbn = {3-540-17746-9},
 year = {1988},
 publisher = {Berlin etc.: Springer-Verlag},
 language = {English},
 zbMATH = {192943},
 Zbl = {0627.30039}
}

@article{epstein94,
 author = {Epstein, David B. A. and Petronio, Carlo},
 title = {An exposition of {Poincar{\'e}}'s polyhedron theorem},
 fjournal = {L'Enseignement Math{\'e}matique. 2e S{\'e}rie},
 journal = {Enseign. Math. (2)},
 issn = {0013-8584},
 volume = {40},
 number = {1-2},
 pages = {113--170},
 year = {1994},
 language = {English},
 zbMATH = {637635},
 Zbl = {0808.52011}
}

@book{bridson_haefliger,
 author = {Bridson, Martin R. and Haefliger, Andr{\'e}},
 title = {Metric spaces of non-positive curvature},
 fseries = {Grundlehren der Mathematischen Wissenschaften},
 series = {Grundlehren Math. Wiss.},
 issn = {0072-7830},
 number = {319},
 isbn = {3-540-64324-9},
 year = {1999},
 publisher = {Berlin: Springer},
 language = {English},
 zbMATH = {1385418},
 Zbl = {0988.53001}
}

@article{akiyoshi21,
 author = {Akiyoshi, Hirotaka and Ohshika, Ken'ichi and Parker, John and Sakuma, Makoto and Yoshida, Han},
 title = {Classification of non-free {Kleinian} groups generated by two parabolic transformations},
 fjournal = {Transactions of the American Mathematical Society},
 journal = {Trans. Am. Math. Soc.},
 issn = {0002-9947},
 volume = {374},
 number = {3},
 pages = {1765--1814},
 year = {2021},
 language = {English},
 doi = {10.1090/tran/8246},
 zbMATH = {7313196},
 Zbl = {1458.57025}
}

@article{futer07b,
 author = {Futer, David},
 title = {Involutions of knots that fix unknotting tunnels},
 fjournal = {Journal of Knot Theory and its Ramifications},
 journal = {J. Knot Theory Ramifications},
 issn = {0218-2165},
 volume = {16},
 number = {6},
 pages = {741--748},
 year = {2007},
 language = {English},
 doi = {10.1142/S0218216507005506},
 zbMATH = {5274997},
 Zbl = {1147.57004}
}

@article{sakuma98,
 author = {Sakuma, Makoto},
 title = {The topology, geometry and algebra of unknotting tunnels},
 fjournal = {Chaos, Solitons and Fractals},
 journal = {Chaos Solitons Fractals},
 issn = {0960-0779},
 volume = {9},
 number = {4-5},
 pages = {739--748},
 year = {1998},
 language = {English},
 doi = {10.1016/S0960-0779(97)00101-X},
 zbMATH = {1370412},
 Zbl = {0934.57011}
}

\end{document}